\documentclass[11pt]{amsart}

\usepackage[english]{babel}
\usepackage[T1]{fontenc} 
\usepackage[utf8]{inputenc} 
\usepackage{datetime}
\usepackage{csquotes} 

\usepackage{geometry} 
\usepackage{hyperref} 

\usepackage{hhline} 
\usepackage{multicol} 
\usepackage{graphicx} 
\usepackage{subcaption} 

\usepackage[citestyle= numeric, bibstyle = numeric]{biblatex}

\usepackage{tikz}
\usetikzlibrary{shapes}
\usetikzlibrary{cd}

\usepackage{maths}

\newtheorem{thm}{Theorem}[section]
\newtheorem{prop}[thm]{Proposition}
\newtheorem{lemma}[thm]{Lemma}
\newtheorem{coro}[thm]{Corollary}

\theoremstyle{definition}

\newtheorem{dfn}[thm]{Definition}
\newtheorem{constr}[thm]{Construction}
\newtheorem{example}[thm]{Example}

\theoremstyle{remark}
\newtheorem{rmk}[thm]{Remark}

\renewcommand{\sf}[1]{\mathsf{#1}} 

\newcommand{\base}{M_{\R}} 
\newcommand{\bbase}{\overline{M}_{\R}} 

\newcommand{\Etrop}{\trop{E}} 

\newcommand{\wall}{\mathrm{Wall}} 

\title[Non-archimedean cylinder counts for blowups of toric surfaces]{Computing non-archimedean cylinder counts in blowups of toric surfaces}
\author{Thorgal HINAULT}
\date{}

\begin{document}

\begin{abstract}
Counts of holomorphic disks are at the heart of the SYZ approach to mirror symmetry. In the non-archimedean framework, these counts are expressed as counts of analytic cylinders.
In simple cases, such as cluster varieties, these counts can be extracted from a combinatorial algorithm encoded by a scattering diagram. The relevant scattering diagram encodes information about infinitesimal analytic cylinders.
In this paper, we give a formula relating counts of a restricted class of cylinders on affine log Calabi-Yau surfaces to counts on a single blowup of a toric surface, using non-archimedean Gromov--Witten theory and a degeneration procedure parametrized by tropical data. This closed-form formula constrasts with the scattering \linebreak diagram approach.
\end{abstract}

\maketitle

\tableofcontents

\section{Introduction}

Some algebro-geometric implementations of the SYZ picture of mirror symmetry consist in constructing explicitly a mirror algebra reflecting the enumerative \linebreak geometry of the initial variety. The way this enumerative data enters in the definition of the mirror algebra is to be thought of as the instanton correction, in a broad sense.

To this day, in the log Calabi-Yau case we have at our disposal essentially two constructions: one using punctured log-Gromov-Witten invariants \cites{abramovich_punctured_2020}{arguz_higher_2020}{gross_intrinsic_2021}{gross_canonical_2021}, and the other relying on non-archimedean enumerative geometry \cites{keelFrobeniusStructureTheorem2019a}{porta_non-archimedean_2022}{porta_non-archimedean_2022-1}{yu_enumeration_2016}{yu_enumeration_2020}{yu_gromov_2015}{yu_tropicalization_2015}. These constructions are known to be equivalent in restricted cases \cite{johnston_comparison_2022}, although a general comparison result between punctured Gromov-Witten theory and non-archimedean Gromov-Witten theory has not been achieved yet.
In both of these constructions, instanton corrections enter as structure constants of the mirror algebra. Structure constants can be expressed in terms of counts of analytic cylinders in the initial variety \cites{gross_canonical_cluster_2017}{keelFrobeniusStructureTheorem2019a}. These analytic cylinders in the non-archimedean setting correspond to the broken lines of \cite{gross_canonical_cluster_2017}.
\\

In this paper, we are interested in computing the non-archimedean cylinder counts for a log Calabi-Yau surface $(\ca{Y},\ca{D})$. We assume that $(\ca{Y},\ca{D})$ is the blow up of a toric surface: this is not a restrictive assumption, as every log Calabi-Yau surface admits a toric model \cite[][Proposition 1.3]{gross_mirror_2015}. The idea is then to relate counts in the blown up variety to counts after we blow down a $(-1)$-curve in the exceptional locus of the toric model. We do it using a deformation procedure parametrized by tropical data, and use analytic geometry to cut out appropriate connected components in the moduli space of non-archimedean stable maps.
Deformation invariant counts are then defined using the powerful formalism of virtual\linebreak fundamental class applied to derived analytic stacks of stable maps \cite{porta_non-archimedean_2022-1}. The upshot is that the geometry of the derived moduli spaces reflects perfectly the axioms of Gromov-Witten theory, so we can handle degenerations of the domain in an easy way. We relate these virtual counts to the naive counts of the mirror construction \cite{keelFrobeniusStructureTheorem2019a} using a smoothness argument, and obtain in this way a formula involving only cylinder counts.

The advantage of this approach is that it leads to explicit, closed-form formula for the counts defining structure constants of the mirror. To the best of our knowledge, this is a new result which contrasts with the scattering diagram\linebreak approach \cite{gross_tropical_2010}, that gives an algorithmic way to compute the structure constants to a finite order. Closed form formulas are of interest as they allow to compute the mirror algebra, give nontrivial relations between various invariants, and enter in the coefficients of wall-crossing functions whose expressions are in most cases conjectural \cite{gross_quivers_2009}. \\

To state the main results of the paper, let $(\ca{Y},\ca{D})$ be a log Calabi-Yau surface with a toric model $\pi\colon (\ca{Y},\ca{D})\rightarrow (\ca{Y}_t,\ca{D}_t)$.
The map $\pi$ is a blowup of non-torus fixed points $x_{ij}$ in the toric boundary, where the $i$ index refers to the toric boundary component, and the $j$ index enumerates the points in the $i$ component.
We denote by $\ca{E}_{ij}$ the irreducible divisor above $p_{ij}$, by $\ca{E}_i$ the union of the exceptional divisors lying above a fixed toric boundary divisor, and by $\ca{E}$ the full exceptional locus.

The formulas involve counts of \emph{primitive (infinitesimal) cylinders} (\S\ref{subsec:tropical-cylinders}). A cylinder is a tropical curve that parametrizes non-archimedean stable maps that (i) meet only two prescribed components of the boundary $\ca{D}$ at a single unspecified point with a given multiplicity, and (ii) have prescribed intersection numbers with each $\ca{E}_i$. We work with primitive cylinders, by which we mean cylinders parametrizing stable maps that have intersection number at most $1$ with each $\ca{E}_i$. To a cylinder $V$ and a curve class $\beta$, we associate a cylinder count $N(V,\beta )$. Our cylinder counts refine the \emph{cylinder spine counts} defined in \cite{keelFrobeniusStructureTheorem2019a}, in the sense that summing our counts over all possible prescriptions in condition (ii) above gives the cylinder spine count.
Given a primitive cylinder $V$, we call the part of the tropical curve parametrizing the intersections with the exceptional divisor the \emph{twig}, and encode the intersection numbers in a tuple of weight vectors called the \emph{twig type}. The length of this tuple is the number of irreducible components of $\ca{E}$ that the associated stable maps meet. In particular twig types of length $1$, as those of the cylinders that appear on the right-hand side of our formula, parametrize stable maps meeting $\ca{E}$ at a single point.

\begin{thm}[{Theorem \ref{theo:reduction-primitive-twig}}]
Let $V$ be a primitive infinitesimal tropical cylinder with twig type $\w = (\w_s)_{1\leq s \leq t}$, let $\beta\in\NE (Y)$. Then
\[N(V, \beta ) = \sum_{\beta_1+\cdots +\beta_t = \beta} \prod_{s=1}^t N(V_s, \beta_s) ,\]
where $V_s$ is an infinitesimal cylinder of twig type $\w_s$ (see Construction \ref{constr:elementary-cylinders}).
\end{thm}

For each component of the toric boundary, let $\ell_i$ denote the number of irreducible components of $\ca{E}_i$. An easy inspection of the curve classes leads to the following refined result:

\begin{coro}[{Corollary \ref{coro:reduction-elementary-cylinder}}]
  Let $V$ be a primitive infinitesimal cylinder of twig type $\w = (\w_s )_{1\leq s\leq t}$. For each $s$, let $\Etrop_{i(s)}$ be the direction of the corresponding twig.

  Then there are at most $\prod_{1\leq s\leq t} \ell_{i(s)}$ curve classes such that $N(V,\beta )\neq 0$. Such a curve class $\beta$ is determined by the choice of an irreducible component $E_{i(s) j}$ for all $s$, and we then have
  \[N(V,\beta ) = \prod_{s=1}^t N(V_s ,\beta_s ),\]
  where $\beta_s$ is the curve class whose intersection number with each irreducible\linebreak component of $E$ is $0$ except for $E_{i(s) j}$.
\end{coro}

This result is a first step towards expressing general cylinder counts (meeting multiple irreducible components of the exceptional divisor) in terms of counts of cylinders touching only one irreducible component of the exceptional divisor. In other words, it expresses primitive cylinder counts of an arbitrary log Calabi-Yau surface in terms of cylinder counts on a toric surface blown up at one point of the toric boundary.

Cylinder counts really depend on the interior of the log Calabi-Yau $(\ca{Y},\ca{D})$. In the case of a single blowup  the interior is $\G_m^2 \cup (\ca{E}\setminus \ca{D})$, where $\ca{E}$ is a $(-1)$-curve and $\ca{D}$ is an snc anticanonical divisor, strict transform of the toric boundary upon taking a toric model. In practice, we can choose any snc compactification of the interior arising from a toric model to compute these counts.
\\

The methods of this paper only work for primitive cylinders, that parametrize stable maps that have simple intersections with the exceptional divisor.
This is because we cannot control the virtual contributions induced by higher intersection numbers solely by tropical means. Concretely, for higher multiplicities the domain curves can have ``bubbles'' mapped to the exceptional divisor, and these are not seen by the tropical picture. Because of this phenomenon, our argument to select connected components in the moduli spaces, which is key to proving deformation invariance of the counts, does not hold anymore.
\\

The paper is organized as follows: in Section \ref{sec:notations} we introduce the geometric set-up, and review non-archimedean Gromov-Witten theory. In Section \ref{sec:trop-geometry} we set up conventions for tropical curves, and define cylinder counts. Section \ref{sec-count-primitive-cyl} is the main body of this work: first we define a tropical deformation, then lift it to a deformation of analytic stable maps, and finally we look at the degeneration of this deformation. We apply the deformation procedure inductively, reducing the number of blowups by one at each step.


\section{Notations and conventions} \label{sec:notations}

Let $(\ca{Y} , \ca{D})$ be a log Calabi-Yau surface over a non-archimedean field $k$ of characteristic $0$. In this section, we fix a toric model and define the associated tropicalization map. After that, we define the relevant moduli spaces of non-archimedean analytic stable maps. Even though we make use of the powerful derived formalism in Section \ref{sec-count-primitive-cyl}, the important Lemma \ref{lemma:underived} allows to identify the virtual counts with the naive counts defined in \cite{keelFrobeniusStructureTheorem2019a}.

\subsection{Geometric setup.}

\subsubsection{Blow-up of toric surfaces.}
Up to applying a toric blowup, we assume that $(\ca{Y} ,\ca{D})$ admits a toric model \cite[][Proposition 1.3]{gross_mirror_2015}. In the counting of stable maps we consider later, applying a toric blowup does not change the counts.

By a toric model, we mean that $(\ca{Y} ,\ca{D})$ is obtained as a sequence of non-toric blow-ups of toric boundary points of a complete smooth toric surface $\ca{Y}_t$.
We choose a cyclic ordering of the irreducible components of the toric boundary ${\ca{D}_t = \sum_{i\in I_{\ca{D}_t}} \ca{D}_{t,i}}$, and denote by $\pi\colon (\ca{Y},\ca{D} )\rightarrow (\ca{Y}_t,\ca{D}_t)$ the toric model. We denote by $T_M = \ca{Y}_t\setminus \ca{D}_t$ the big torus with cocharacter lattice $M$.

More precisely, $(\ca{Y},\ca{D})$ is obtained from $(\ca{Y}_t,\ca{D}_t)$ as follows: fix a tuple $(\ell_1, \dots ,\ell_m)$ of integers, and over each divisor $\ca{D}_{t,i}$ choose $\ell_i$ distinct points $x_{ij}$, not lying in any $0$-strata of $\ca{D}_t$.
Let $\pi\colon \ca{Y} {\rightarrow} \ca{Y}_t$ be the blowup at these points.
We denote by $\ca{D}_i$ the strict transform of $\ca{D}_{i,t}$, and by $\ca{E}_{ij}$ the exceptional divisor lying above $x_{ij}$. We also set $\ca{D} = \sum_{i} \ca{D}_i$ and $\ca{E} = \sum_{i,j} \ca{E}_{ij}$.
We thus have the relation
\[c_1 (T_{\ca{Y}}) = \pi^{\ast} c_1 (T_{\ca{Y}_t}) - \ca{E} = \ca{D}-\ca{E}.\]
We denote by $\Sigma$ the fan associated to $\ca{Y}_t$.

\subsubsection{Tropicalization through toric model.}
From now on, we work with Berkovich analytifications which we denote with straight letters: ${Y=  \an{\ca{Y}}}$, ${D=\an{\ca{D}}}$, \linebreak${Y_t = \an{\ca{Y}_t}}$, ${D_t =\an{\ca{D}_t}}$, and so on.

Since $U = Y\setminus D$ is log Calabi-Yau, we can consider its essential skeleton $\Sk (U)$, which comes with an integral affine structure. It is constructed from the log Calabi-Yau volume form on $U$ using Temkin's Kähler seminorm. Similarly, we consider the skeleton $\Sk (T_M)$ with its canonical integral affine structure. There is an isomorphism $\Sk (T_M )\simeq \base = M\otimes \R$, compatible with integral affine structures meaning that $\Sk (T_M , \Z ) \simeq M$. Up to applying this isomorphism, we assume $\Sk (T_M )  =\base$.

\begin{rmk} \label{rem:norm-in-M}
The skeleton $\Sk (U)$ is naturally included inside the Clemens polytope of $(Y,D)$, inducing an embedding $\Sk (U)\hookrightarrow \R_{\geq 0}^{m}$ where $m$ is the number of components of $D$. Integral points of the skeleton have coordinates in $\Z_{\geq 0}^m$ under this inclusion, and using the identification $M\simeq \Sk (U,\Z )$ the \emph{norm} of a vector $v\in M$ is defined as the sum of the absolute value of its coordinates under the embedding $M\hookrightarrow \Z^m$.
\end{rmk}

We denote by $\bbase$ the natural compactification of the essential skeleton induced by the $\G_m^2\hookrightarrow Y_t$ (\emph{cf.} Figure \ref{fig:skeleton} for $\G_m^2\hookrightarrow \mathbb{P}^2$). We consider the tropicalization of $U$ through the toric model $\pi$, and extend it to the compactification $Y$
\[\tau\colon Y\overset{\pi}{\longrightarrow} Y_t\overset{\tau_t}{\longrightarrow} \bbase .\]
We refer to $\tau$ as the tropicalization map.

\begin{figure}
  \begin{tikzpicture}

\begin{scope}[scale=1]

\draw (2,0) -- (0,0) -- (0,2) ;
\draw (0,0) -- (-1.4,-1.4) ;

\end{scope}

\begin{scope}[scale=1, shift={(8,0)}]

\draw (2,0) -- (0,0) -- (0,2) ;
\draw (0,0) -- (-1.4,-1.4) ;


\draw (2,2) -- (2,-4.8) -- (-4.8,2) -- (2,2) ;

\end{scope}

\end{tikzpicture}
  \caption{The fan of $\mathbb{P}^2$ drawed in $\Sk (\G_m^2 )$, and the natural compactification $\overline{\Sk (\G_m^2)}$.}
  \label{fig:skeleton}
\end{figure}
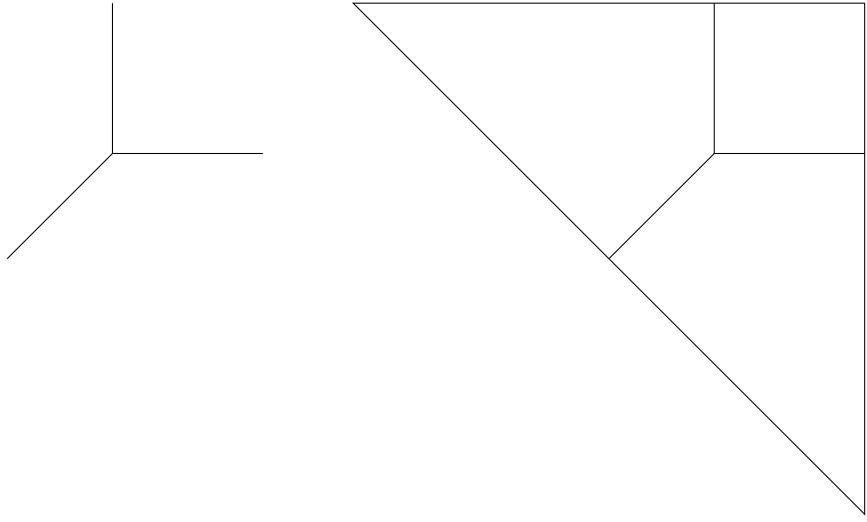

\subsection{Non-archimedean stable maps}
In this paper, we work with \linebreak non-archimedean stable maps, and non-archimedean Gromov-Witten invariants. We can always recover algebraic statements using GAGA theorems for\linebreak non-archimedean analytic stacks, given that $Y$ is proper \cite{porta_higher_2016}.

We refer to \cites{porta_non-archimedean_2022}{porta_non-archimedean_2022-1} for the general theory of stable maps in non-archimedean geometry. Below we recall the main definitions and results.

\subsubsection{Derived non-archimedean stable maps.}

\begin{dfn}
  Let $S$ be a rigid $k$-analytic space, let $X\rightarrow S$ be a smooth rigid $k$-analytic space over $S$. We denote by $\overline{M} (X/S,\tau , \beta )$ the moduli space of $(\tau ,\beta )$-stable maps.

  In the special case when $\tau$ is an $n$-valent vertex and $\beta$ has genus $0$, we denote this moduli space by $\overline{M}_{0,n} (X/S ,\beta )$.
\end{dfn}

\begin{thm}[{\cite[][Theorems 1.1, 1.2]{porta_non-archimedean_2022-1}}]
  Let $S$ be a rigid $k$-analytic space and let $X\rightarrow S$ be a smooth rigid $k$-analytic space over $S$. Let $(\tau ,\beta )$ be an $A$-graph. Then:
   \begin{enumerate}
     \item The moduli stack $\overline{M} (X,\tau ,\beta )$ of $(\tau ,\beta )$-stable maps admits a derived\linebreak enhancement $\R\overline{M} (X/S , \tau ,\beta )$ that is a derived $k$-analytic stack, locally of finite presentation and derived lci over $S$.
     \item If $S$ is an algebraic variety and $X$ is an algebraic variety over $S$, then \[\an{\R\overline{M} (X/S , \tau ,\beta) }\overset{\sim}{\rightarrow} \R\overline{M} (\an{X} /\an{S} , \tau ,\beta ) .\]
     \item The derived moduli stacks $\R\overline{M} (X/S,\tau,\beta )$ satisfy a list of geometric relations reflecting the Behrend-Manin axioms of Gromov-Witten theory.
   \end{enumerate}
\end{thm}

We denote by $\trunc_0$ the truncation functor, so that $\overline{M} (X,\tau ,\beta ) = \trunc_0 \R\overline{M} (X/S , \tau ,\beta ) $.

\subsubsection{Relative derived non-archimedean stable maps.}

Higher tangency conditions at the $i$-th marked point can be considered in the derived theory, using infinitesimal thickenings of the domain curve along the sections \cites[][\S9]{porta_non-archimedean_2022}[][\S8]{porta_non-archimedean_2022-1}.
Given a contact order $m_i\in\N_{>0}$, evaluation maps with multiplicity are constructed
\[\ev_i^{m_i}\colon \R \overline{M} (X,\tau ,\beta )\rightarrow X_{i,\tau}^{m_i} .\]
Using this map, one can parametrize stable maps with contact order $m_i$ along a lci closed analytic subspace $\iota\colon Z_i\hookrightarrow Y$ at the $i$-th marked point by considering the substack given by the fiber product
\[\R \overline{M} (Y,\tau ,\beta ) \times_{X_{i,\tau}^{m_i}} Z_{i,\tau}^{m_i} ,\]
where $Z_{i,\tau}^{m_i}$ is obtained from $Z_i$ using the same procedure of thickening along the $i$-th section.

We use these evaluation maps with multiplicities to construct derived versions of the moduli spaces considered in \cite{keelFrobeniusStructureTheorem2019a}. Let $J$ be a finite set of cardinality $n$, and let $\p =(\p_j )_{j\in J}$ be a tuple of points in $\Sk (U,\Z )$. Recall that points in $\Sk (U,\Z )$ are valuations on $k(U)$ with integral values on $k^0 (U^0)$. Define
\[B = \lrbrace{j\in J\;\vert\; \p_j\neq 0} \text{ and } I = \lrbrace{j\in J\;\vert\; \p_j = 0}  .\]
For $j\in B$, we write $\p_j = m_j\nu_j$ where $\nu_j$ is a divisorial valuation with divisorial center $D_j\subset D$ and $m_j\in\N_{>0}$.
Given $\beta\in\NE (\ca{Y} )$, we define a sequence of derived moduli spaces:
\[\R\sm{M } (U,\p,\beta )\subset\R\sd{M} (U,\p ,\beta )\subset \R M (U,\p ,\beta )\subset \R\overline{M} (Y,\p ,\beta ) \subset \R\overline{M} (Y,\beta ).\]
These moduli spaces are defined as follows:
\begin{itemize}
  \item $\R\overline{{M}} (Y,\p, \beta )$ corresponds to stable maps $[C ,(p_j)_{j\in J} , f\colon C\rightarrow Y]$ such that for every $j\in B$, $p_j$ is mapped to $D_j$ with multiplicity at least $m_j$.
  \item $\R {M} (U,\p, \beta )$ corresponds to the substack of stable maps \linebreak${[C,(p_j)_{j\in J}, f\colon C\rightarrow Y ]}$ such that $p_j$ is mapped to the open stratum $D_j^{\circ}$ for all $j\in B$, and $f^{-1} (D)= \sum_{j\in B} m_j p_j$.
  \item $\R\sd{{M}}(U,\p ,\beta )$ corresponds to the substack of stable maps with stable domain curve.
  \item $\R\sm{{M}} (U,\p ,\beta )$ corresponds to the substack of stable maps\linebreak ${[C,(p_j)_{j\in J} , f\colon C\rightarrow Y]}$ such that:
  \begin{enumerate}
    \item $f^{\ast} (T_Y (-\log D))$ is a trivial vector bundle on $C$.
    \item $f(C)\cap (D\cap E) = \emptyset$.
    \item $f^{-1} (E)$ is a finite set of points without multiplicities, disjoint from the nodes and the marked points of $C$.
  \end{enumerate}
\end{itemize}
We also consider the underived moduli spaces
\[\overline{M} (Y,\p ,\beta ) = \trunc_0 \R \overline{M} (Y, \p ,\beta ) ,\]
and so on, which agree with the moduli spaces defined in \cite[][\S 3]{keelFrobeniusStructureTheorem2019a}
Note that all of these stacks are analytification of the corresponding algebraic versions, that the three leftmost inclusions are (Zariski) open, and that $\sd{M} (U,\p ,\beta )$ and $\sm{{M}} (U,\p,\beta )$ are varieties since we only consider rational curves.

\subsubsection{Non-archimedean Gromov-Witten invariants.}

Two conditions are needed to define numerical Gromov-Witten invariants: properness of the moduli space, to get a pushforward to a point, and a virtual fundamental class to cap cycles on the moduli space with. In the non-archimedean theory, rigid motivic Borel-Moore homology is used and a virtual fundamental class $[X/S]\in \BMho[d] (X/S , \Q_S (2d ))$ is associated to any derived lci morphism of derived analytic stack $\varphi\colon X\rightarrow S$ of virtual dimension $d$ \cite[][Definition 4.4]{porta_non-archimedean_2022}.

\begin{thm}[{\cite[][Theorem 1.1]{porta_non-archimedean_2022}}]
  Let $S$ be a rigid $k$-anaytic space and let $X\rightarrow S$ be a rigid $k$-analytic space smooth over $S$. Let $(\tau ,\beta )$ be an $A$-graph.
  \begin{enumerate}
    \item There exists a virtual fundamental class
    \[[\R\overline{M} (X/S , \tau ,\beta )]\in\BMho[d] (\R\overline{M} (X/S ,\tau ,\beta ) /S , \Q_S (2d)) ,\]
    where $d$ is the virtual dimension.
    \item The system of virtual fundamental classes satisfies the Behrend-Manin axioms of Gromov-Witten theory.
  \end{enumerate}
\end{thm}
Take $S=  \Spf k$ and assume $k$ has characteristic $0$. Given an $A$-graph $(\tau ,\beta )$, the virtual fundamental class, lci closed subvarieties $Z_i\subset X$ with contact order $m_i\in\N_{>0}$ and the associated diagrams
\[
\begin{tikzcd}
  \R \overline{M} (X,\tau ,\beta ) \ar[r, "\ev_i^{m_i}"] \ar[d, "\st"] & X_{i,\tau}^{m_i} & Z_{i,\tau}^{m_i} \ar[l] \\
  \overline{M}_{\tau}
\end{tikzcd}
\]
one can define numerical Gromov-Witten classes and associated numerical\linebreak invariants using the usual procedures \cite[][Definition 8.1]{porta_non-archimedean_2022}. That the subvarieties $Z_i$ be lci is crucial to be able to define the virtual fundamental class.
These classes satisfy the Behrend-Manin axioms of Gromov-Witten theory as a consequence of these same axioms for the derived moduli spaces and the functoriality properties of the virtual fundamental class.

We need the following lemma to define the relevant numerical Gromov-Witten invariants later.
\begin{lemma}\label{lemma:underived}
  The derived moduli stack $\R \sm{M} (U,\p ,\beta )$ is underived, meaning we have a canonical equivalence
  \[\sm{M} (U,\p,\beta ) \overset{\sim}{\longrightarrow} \R \sm{M} (U,\p ,\beta ).\]
  In particular, it is smooth over $\overline{M}_{0,n}$.
\end{lemma}
\begin{proof}
  By \cite[][Lemma 3.6]{keelFrobeniusStructureTheorem2019a}, the moduli space $\sm{M} (U,\p ,\beta )$ is smooth over $\overline{M}_{0,n}$, thus its dimension equals the virtual dimension of $\R \sm{M} ( U,\p ,\beta )$.
  By \cite[][Proposition 2.14]{porta_non-archimedean_2022}, we deduce that the canonical closed immersion
  \[\sm{M} (U,\p ,\beta )\hookrightarrow \R\sm{M} (U,\p ,\beta )\]
  is an equivalence.
\end{proof}


\section{Tropical curves} \label{sec:trop-geometry}

In this section, we review the notion of spines, tropical curves and twigs in the affine manifold $\base$. In Construction \ref{constr:topology-trop-curves}, we define an explicit topology on the space of tropical curves. Then we define the \emph{tropical cylinders}, which are tropicalizations of the analytic cylinders we want to count, and the associated counts. These refine the spine counts defined in \cite{keelFrobeniusStructureTheorem2019a}. In order to keep track of the combinatorics, we define the notion of twig type associated to a cylinder.

\subsection{Space of tropical curves}
We will consider tropical curves in the $\Z$-affine manifold $M_{\R}$, which we refer to as the tropical base. These tropical curve will be used to parametrize analytic stable maps in $\sm{M} (U,\p ,\beta )$, so we require more than the usual balancing condition in their definition. Rather than giving precise definitions, which are spelled out in \cite[][\S 4]{keelFrobeniusStructureTheorem2019a}, we illustrate the relevant notions in the $2$-dimensional case through concrete examples.

\begin{example}\label{example:P2}
We choose as our running example the log Calabi-Yau surface $(Y,D)$ obtained from $\mathbb{P}^2$ by blowing up the three toric divisors at $2$ points. We denote by $D_1$, $D_2$ and $D_3$ the irreducible components of $D$, and by $E_{11}$, $E_{12}$, $E_{21}$, $E_{22}$, $E_{31}$ and $E_{32}$ the exceptional divisors.
\end{example}

\subsubsection{The canonical wall structure. }
 The tropical base carries a canonical wall structure denoted by $\wall = \bigcup_{n\geq 0} \wall^n$, which is essentially a collection of codimension $1$ integral cones with an attached wall-crossing function constructed inductively in a combinatorial way. Concretely, in the $2$-dimensional case, walls are rays starting from the origin of $\base$. The wall-crossing functions will not be considered in this article, so we omit them in the following outline of the construction (see \cite[][Construction 4.16]{keelFrobeniusStructureTheorem2019a}):
 \begin{itemize}
   \item Initial walls $\wall^0$: let $\trop{E}_{ij} \coloneqq \tau (E_{ij})$ be the tropicalization of the \linebreak exceptional divisor $E_{ij}$. It is a point in $\partial \bbase$ at the end of the ray corresponding to $D_i$. We also denote by $\trop{E}$ the union of $\trop{E}_{ij}$, so $\wall^0$ corresponds to rays of the fan of $Y_t$ that contain points of $\trop{E}$.
   \item $\wall^{n+1}$ from $\wall^n$: add to $\wall^n$ the rays generated by sums of two vectors in walls of $\wall^n$.
 \end{itemize}

\begin{example}
  For our running example $(Y,D)$, Figure \ref{fig:wall-structure-P2} shows $\wall^3$.

  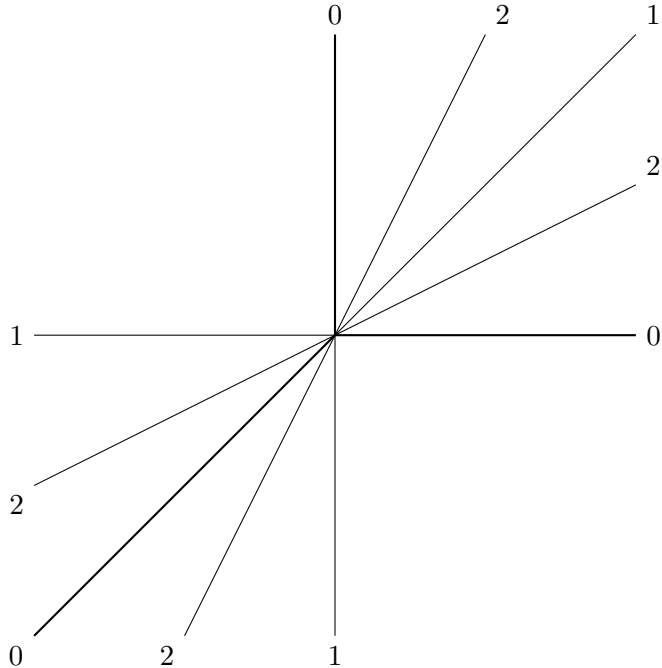
\begin{figure}
    \begin{tikzpicture}
\begin{scope}[scale=0.5]
\draw[thick] (8,0) node[right] {$0$}-- (0,0) -- (0,8) node[above]{$0$};
\draw[thick] (0,0) -- (-8,-8) node[below left]{$0$};

\draw (-8,0) node[left]{$1$} -- (0,0) -- (0,-8) node[below] {$1$};
\draw (0,0) -- (8,8) node[above right] {$1$};

\draw (8,4) node[above right]{$2$} -- (0,0) -- (4,8)node[above right]{$2$} ;
\draw (-8,-4) node[below left] {$2$}-- (0,0) -- (-4,-8) node[below left] {$2$};

\end{scope}

\end{tikzpicture}
    \caption{The $\wall^3$ part of the wall structure of Example \protect\ref{example:P2} (obtained in three steps). The numbers on the rays indicate at which step of the construction a ray was added, thick rays are initial walls.}
    \label{fig:wall-structure-P2}
  \end{figure}

\end{example}

The wall structure is a necessary ingredient in our notion of tropical curves. More importantly, it carries a lot of geometric information as the structure \linebreak constants of the mirror algebra can be obtained by analyzing the interaction between tropical cylinders and walls. We note that by construction, in the toric case there are no walls.

\subsubsection{Nodal metric trees}

A metric tree is a finite abstract tree $\Gamma$ together with an identification of every edge $e$ with a closed interval in $\intervalleff{0}{\ell}$ where $\ell\in\intervalleof{0}{+\infty}$. We refer to \cite{keelFrobeniusStructureTheorem2019a} for the notions of infinite and finite vertices, leg, node, irreducibility and stability of metric trees.

Let $J$ be a finite set of cardinality $n$, a metric tree with $n$ legs $[\Gamma ,(v_j)_{j\in J}]$ is a nodal metric tree $\Gamma$ with $1$-valent vertices $(v_j)_{j\in J}$ and no other $1$-valent vertices. It is called \emph{extended} if every $v_j$ is infinite. It is called \emph{simple} if there are no finite $2$-valent vertices.
We denote by $F\subset J$ the subset of indices corresponding to finite legs. Given a pointed tree $[\Gamma ,(v_j)_{j\in J}]$, we will frequently denote by $P_{ij}$ the path in $\Gamma$ connecting the marked points $v_i$ and $v_j$ for $i,j\in J$.

Nodal metric trees will be used as domains of tropical curves, spines or twigs to $\base$, notions which we now define.

\subsubsection{Tropical curves, spines and twigs}
We refer the reader to Figure \ref{fig:example-tropical-curves} for illustrations of the following notions.

Let $J$ be an indexing set as above. We will always fix a partition $I\coprod B = J$ and a subset $F\subset J$. In particular, we allow $F = \emptyset$ and systematically omit it from the notations in this case. In the following definitions, when $F= \emptyset$ we call the objects \emph{extended}.

Given a $\Z$-affine immersion $h\colon \Gamma\rightarrow \bbase$ from a nodal metric tree to $\bbase$, the slope of $h$ at a vertex $v$ of $\Gamma$ along an edge $e\in\Gamma$ is an integral vector which we call the \emph{weight vector} of $h$ at $v$ in the direction $e$, and denote by $w_{(v,e)}$.
The \emph{degree} of $h$ at $v$ along the edge $e$ is the norm of $w_{(v,e)}$, in the sense of Remark \ref{rem:norm-in-M}.

\begin{rmk}
  If $v$ is a vector parallel to the direction of a ray of the fan $\Sigma_t$, then the norm of $v$ is the index of $v$ in $M$ (that is, the least common multiple of its coordinates).
\end{rmk}

A \emph{tropical curve in $\base$} is a $\Z$-affine immersion $T= [\Gamma ,(v_j)_{j\in J} , h\colon \Gamma\rightarrow\bbase ]$ from an extended stable simple nodal metric tree to $\bbase$ that is balanced at every vertex of valency greater than $1$, constant on the $v_i$-leg for every $i\in I$, has weight vector on each $v_j$-leg in the direction of a ray of the fan if $j\in B$, and such that every infinite leg not labeled by a marked point is mapped to $\trop{E}$. \\
We denote by $\p = ( \p_j)_{j\in J}$ the tuple of vectors given by the slope of $h$ along the marked legs, and say that $T$ is a tropical curve \emph{of type $\p$}. Note that by definition $\p_j$ is non zero if and only if $j\in B$.
These tropical curves parametrize stable maps in $\sm{M} (U,\p ,\beta )$, and their space is denoted by $\TC(\base , \p)$. \\

A \emph{spine in $\base$} is a $\Z$-affine immersion $S = [\Gamma , (v_j)_{j\in J} , h\colon \Gamma\rightarrow \bbase]$ from a stable nodal metric tree to $\bbase$ with legs indexed by $J$, whose image meets $\partial\bbase$ precisely at the marked points indexed by $B$ and such that the sum of the weight vectors at each vertex $v$ of valency greater than $1$ is either $0$ ($h$ is balanced at $v$) or is contained in a wall ($v$ is a bending vertex). Furthermore, marked points indexed by $F$ correspond to finite legs.\\
If we denote by $\p = (\p_j)_{j\in J}$ the tuple of weight vectors of $h$ along the legs, we say that $S$ is a spine of type $\p^F$.
Spines parametrize restriction of analytic stable maps to the convex hull of the marked points in the domain curve. In particular, fixing a spine does not specify how the associated analytic stable maps meet the exceptional divisor. The space of spines of type $\p^F$ is denoted by $\SP (\base , \p^F)$. \\

A \emph{twig in $\base$} is a $\Z$-affine immersion $[\Gamma , (r,u_1 , \dots , u_t), h\colon \Gamma\rightarrow \bbase]$ where $\Gamma$ is a nodal metric tree, all the legs are marked and only the $r$-leg is finite, the image of $\Gamma$ is contained in $\wall$, each $u_i$ is mapped to $\trop{E}$, and $h$ is balanced at every vertex of valency greater than $1$. We refer to $r$ as the \emph{root} of the twig, and to the $u_i$ as the \emph{leaves}. \\
The weight vector of $h$ at $r$ is the \emph{direction} of the twig, and the (ordered) tuple of weights of $h$ at each leaves is called the \emph{combinatorial type} of the twig.

\begin{rmk}
  These notions play well together, in the sense that given a tropical curve $[\Gamma , (v_j)_{j\in J} , h\colon\Gamma\rightarrow \bbase]$ of type $\p$, if we denote by $\Gamma^s$ the convex hull of the marked points in $\Gamma$ then:
  \begin{enumerate}
    \item $[\Gamma^s , (v_j)_{j\in J} , h_{\vert \Gamma^s}]$ is a spine of type $\p$ \cite[][Lemma 4.23]{keelFrobeniusStructureTheorem2019a}.
    \item The restriction of $h$ to the closure of connected components of $\Gamma\setminus\Gamma^s$ are twigs.
  \end{enumerate}
  In particular, we have \emph{spine map}
  \[\Sp \colon \TC (\base ,\p ) \rightarrow \SP (\base ,\p ).\]
\end{rmk}

\begin{example}
  In Figure \ref{fig:example-tropical-curves} several tropical curves are drawn in the skeleton. They illustrate the general fact that many twigs are compatible with a given spine, and that different twigs associated to a spine can have varying number of leaves. Furthermore, we can very often vary the degree of the leaves in such a way that the ``shape'' of the twig is invariant, but the degrees of the leaves become very large. For example, the shape of the twig in the bottom right corner is realized by the twig types $\lrbrace{(2+n , 0) , (1+n , 0) , -(n,n)}$ for every $n\in\N$, where the degrees are $2+n$, $1+n$ and $n$.
\end{example}

\begin{figure}
\begin{subfigure}{\textwidth}
  \begin{tikzpicture}
\begin{scope}[scale=0.35]
\draw (8,0) node[right] {$\trop{E}_1$} -- (0,0) -- (0,8) node[above] {$\trop{E}_2$} ;
\draw (0,0) -- (-8,-8) node[below left] {$\trop{E}_3$};

\draw[ultra thin] (-8,0) -- (0,0) -- (0,-8) ;

\draw[ultra thin] (0,0) -- (8,8) ;

\draw[ultra thin] (8,4) -- (0,0) -- (4,8);
\draw[ultra thin] (-8,-4) -- (0,0) -- (-4,-8) ;

\draw[blue, thick] (2,8) node[above] {$1$} -- (2,2) -- (8,2) node[right] {$1$} ;

\draw[red, thick] (2,2) -- (0,0) -- (-8,-8) node[above] {$1$} ;

\end{scope}

\begin{scope}[scale=0.35,shift={(20,0)}]
\draw (8,0) node[right] {$\trop{E}_1$} -- (0,0) -- (0,8) node[above] {$\trop{E}_2$} ;
\draw (0,0) -- (-8,-8) node[below left] {$\trop{E}_3$};

\draw (-8,0) -- (0,0) -- (0,-8) ;

\draw (0,0) -- (8,8) ;

\draw (8,4) -- (0,0) -- (4,8);
\draw (-8,-4) -- (0,0) -- (-4,-8) ;

\draw[blue, thick] (2,8) node[above] {$1$} -- (2,2) -- (8,2) node[right] {$1$} ;

\draw[red, thick] (2,2) -- (0,0) -- (-8,-8) node[above] {$2$};
\draw[red, thick] (8,0) node[above]{$1$} -- (0,0) -- (0,8) node[left] {$1$};

\end{scope}
\end{tikzpicture}
\end{subfigure}

\begin{subfigure}{\textwidth}
  \begin{tikzpicture}
\begin{scope}[scale=0.35]
\draw (8,0) node[right] {$\trop{E}_1$} -- (0,0) -- (0,8) node[above] {$\trop{E}_2$} ;
\draw (0,0) -- (-8,-8) node[below left] {$\trop{E}_3$};

\draw[ultra thin] (-8,0) -- (0,0) -- (0,-8) ;

\draw[ultra thin] (0,0) -- (8,8) ;

\draw[ultra thin] (8,4) -- (0,0) -- (4,8);
\draw[ultra thin] (-8,-4) -- (0,0) -- (-4,-8) ;

\draw[blue, thick] (-4,8) node[left] {$3$ }-- (-4 ,-2) -- (-8,-6) node[left] {$4$};

\draw[red, thick] (-4,-2) -- (0,0) -- (0,8) node[left] {$1$} ;
\draw[red, thick] (0,0) -- (8,0) node[above] {$2$};

\end{scope}

\begin{scope}[scale=0.35,shift={(20,0)}]
\draw (8,0) node[right] {$\trop{E}_1$} -- (0,0) -- (0,8) node[above] {$\trop{E}_2$} ;
\draw (0,0) -- (-8,-8) node[below left] {$\trop{E}_3$};

\draw (-8,0) -- (0,0) -- (0,-8) ;

\draw (0,0) -- (8,8) ;

\draw (8,4) -- (0,0) -- (4,8);
\draw (-8,-4) -- (0,0) -- (-4,-8) ;

\draw[blue, thick] (-4,8) node[left] {$3$ }-- (-4 ,-2) -- (-8,-6) node[left] {$4$};

\draw[red, thick] (-4,-2) -- (0,0) -- (0,8) node[left] {$2$} ;
\draw[red, thick] (0,0) -- (8,0) node[above] {$3$};
\draw[red, thick] (0,0) -- (-8,-8) node[above] {$1$};

\end{scope}
\end{tikzpicture}
\end{subfigure}
  \caption{Tropical curves in $\base$ for Example \protect\ref{example:P2}. The spines are drawn in blue, and the twigs in red. The numbers correspond to the degree of the $\Z$-affine immersion along the legs.}
  \label{fig:example-tropical-curves}
\end{figure}
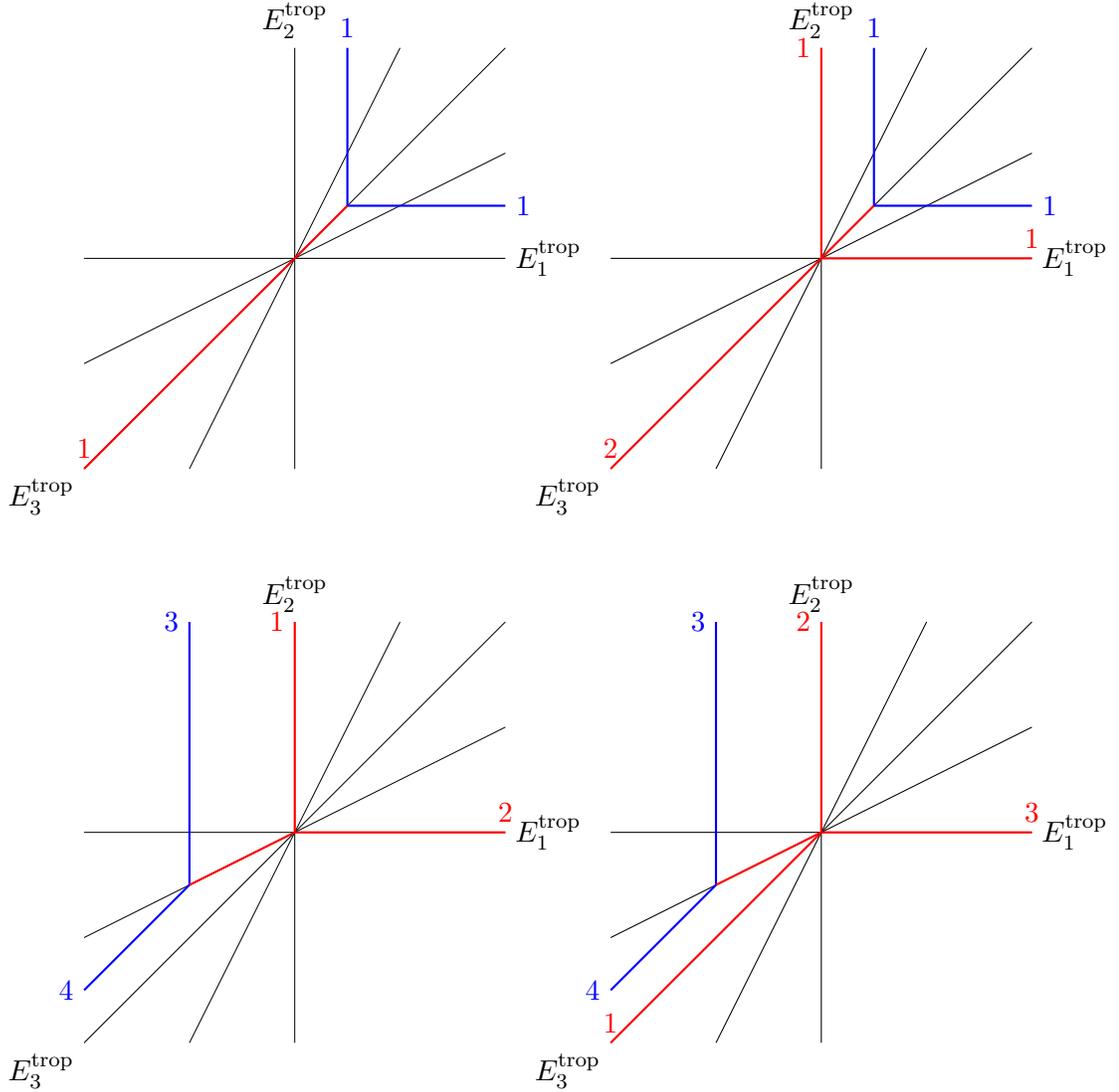

In practice to define a spine, it is enough to specify its behaviour around vertices. We can then recover an extended spine by extending the map using the $\Z$-affine structure of $\base$. This is made precise in the following construction, that we explicit mostly to set up notations about curve classes.

\begin{constr}\label{constr:extension-spines}
  Given an unextended spine $S = [\Gamma ,(v_j)_{j\in J} , h]$ of type $\p^F$, the associated extended spine $\widehat{S} = [\widehat{\Gamma}  ,(\widehat{v}_j)_{j\in J} , \widehat{h}]$ is the spine obtained by applying the following procedure for each $j\in F$:
  \begin{enumerate}
    \item Glue a copy of $\ell_j\coloneqq \intervalleff{0}{+\infty}$ at $v_j$, and replace the marked point $v_j$ by $\widehat{v}_j = \infty$ the infinite endpoint of $\ell_j$.
    \item Extend $h$ affinely on $\ell_j$ with slope $\p_j$.
  \end{enumerate}
  To each new leg, we can associate a curve class $\delta_j\in\NE (Y)$ using a piecewise-linear function $\varphi\colon M_{\R}\rightarrow N_1 (Y_t, \R )$. The curve class corresponding to the extension from $S$ to $\widehat{S}$ is then defined as $\widehat{\delta} = \sum_{j\in F} \delta_j$. Given $\beta\in\NE (Y)$, the associated extended curve class is $\widehat{\beta}\coloneqq \beta +\widehat{\delta}$.
\end{constr}

\subsubsection{Topology on $\TC (\base , \p)$}
The last ingredient we will need to parametrize the deformation procedure at the tropical level is a topology on $\TC (\base ,\p )$. We define it explicitly by giving a basis of open neighbourhoods, that we will use in \S \ref{subsec:tropical-deformation} to prove that we select connected components in the relevant spaces of tropical curves. This topology was considered in the first version of \cite{keelFrobeniusStructureTheorem2019a}, where it is proved that it is Hausdorff and that the natural tropicalization maps $\sm{M} (U,\p ,\beta )\rightarrow \TC (\base , \p) $ are continuous. The topology is essentially given by deformations of the domain and of the image of tropical curves.

\begin{constr}\label{constr:topology-trop-curves}
  Let $T= [\Gamma ,(v_j)_{j\in J} , h] \in\TC (\base ,\p )$, let $\varepsilon >0$ and let ${\ca{W} = (W_k)_{k\in K}}$ be an open covering of $\bbase$.
  Define a basic open neighbourhood $U\lr{T ,\varepsilon , \ca{W}}$ of $T$ as the set of tropical curves $[\Gamma ' ,(v_j' )_{j\in J} , h']$ that satisfy:
  \begin{enumerate}
    \item There is a continuous map $c\colon \Gamma '\rightarrow \Gamma$ contracting a subset of topological edges of $\Gamma '$, sending each $v_j'$ to $v_j$ and each node of $\Gamma'$ to a node of $\Gamma$.
    \item The sum of the length of all edges contracted by $c$ is less then $\varepsilon$.
    \item For each edge $e$ of $\Gamma$, let $e'$ be the edge of $\Gamma'$ such that $c(e')=e$. If $e$ has finite length, then the difference between the lengths of $e$ and $e'$ is less than $\varepsilon$. If $e$ has finite length, then the length of $e'$ is greater than $1/\varepsilon$.
    \item For each vertex $v'$ of $\Gamma '$, if $h(c(v'))\in W_k$ then $h'(v')\in W_k$.
    \item For each edge $e'$ of $\Gamma '$ not contracted by $c$, the derivative of $h$ on $c(e')$ is equal to the derivative of $h'$ on $e'$.
  \end{enumerate}
\end{constr}

\subsection{Tropical cylinders} \label{subsec:tropical-cylinders}

In this subsection, we refine the infinitesimal cylinder counts of \cite{keelFrobeniusStructureTheorem2019a} based on the combinatorial structure of twigs.

\begin{dfn}
Let $\p =(\p_1,\p_2 )$ be a vector of weights, such that $\p_1 + \p_2$ is parallel to the direction of a wall.
\begin{enumerate}
  \item An \emph{(infinitesimal) cylinder spine} of type $\p$ is a spine $S = [\intervalleff{-\varepsilon}{\varepsilon} , (v_1,v_2) , h]$  of type $\p$, where $0$ is the only bending vertex, and where $\varepsilon >0$ is chosen so that the image of $S$ intersects a single wall.
  \item An \emph{(infinitesimal) tropical cylinder} of type $\p$ is a balanced pointed tree in $\base$ obtained from an infinitesimal cylinder spine of type $\p$ by adding a single twig, and gluing an infinite constant leg to a point of the spine distinct from the bending vertex.
\end{enumerate}
  By construction, infinitesimal tropical cylinders have a single twig. We refer to the type of this twig as the \emph{twig type of the cylinder}. We call a tropical cylinder \emph{primitive} if the degree of every leg of the twig is equal to $1$ and all the legs have a different direction.

  Given an infinitesimal cylinder $V$, we define the associated extended tropical cylinder $\widehat{V}$ to be the tropical curve obtained after applying the extension procedure of Construction \ref{constr:extension-spines} to $\Sp (V)$.
\end{dfn}

\begin{example}
All the tropical curves in Figure \ref{fig:example-tropical-curves} are extended tropical cylinders. If we restrict the spine to a region of the domain such that the image only meets $\wall$ at the bending vertex, then we obtain infinitesimal cylinders. Only the top right tropical cylinder is primitive.
\end{example}

\begin{dfn} \label{dfn:spine-count}
  Given $\beta\in\NE (Y)$ and a spine $S$ of type $\p^F$, the count $N(S,\beta )$ is the length of a certain subset $F_w (S_w ,\beta)\subset {M}_{0,n} (U,\p ,\widehat{\beta})$, where $S_w$ is obtained from $S$ by gluing an interior leg to a point of $S$ distinct from the bending vertex. The subset is constructed in three steps:
  \begin{enumerate}
    \item Extend the spine to $\widehat{S}_w$.
    \item Look at $\ev_w^{-1} (h(w))$ in $\sm{{M}}_{0,n} (U ,\p , \widehat{\beta})\cap \Sp^{-1} (\widehat{S}_w)$.
    \item Consider the subset of stable maps that satisfy the toric tail condition \cite[][Construction 9.3]{keelFrobeniusStructureTheorem2019a}.
  \end{enumerate}
  It is proved in \cite[][Proposition 9.5]{keelFrobeniusStructureTheorem2019a} that this count is independent of the choice of $w$, by showing that it is the degree of $(\st ,\ev_w)$ restricted to some closed subset of the target. In particular, for cylinder spines, the count is just the degree of some restriction of the evaluation map at $w$.
\end{dfn}

We now define the counts associated to an infinitesimal tropical cylinder, they refine the cylinder spine count.

\begin{dfn}\label{dfn:cylinder-count}
  Let $V$ be a tropical cylinder (infinitesimal or extended), and let $\beta\in\NE (Y)$. Let $S= \Sp (V)$ be the associated cylinder spine, and let $S_w$ be as in Definition \ref{dfn:spine-count}. The \emph{count of tropical cylinders} associated to $V$ is
  \[N (V ,\beta ) \coloneqq \length\lr{ F_w (S_w ,\beta )\cap \Trop^{-1} (\widehat{V})} .\]
\end{dfn}

\begin{rmk}
  If $V$ is extended, then $\widehat{V} = V$ and $\widehat{\beta} = \beta$. If $V$ is not extended, then $N(V,\beta ) = N(\widehat{V} , \widehat{\beta})$.
\end{rmk}

\begin{rmk}
Given a cylinder spine $S$ and a mark $w$, we have
\[N(S,\beta ) =\sum_{V\in \Sp^{-1} (S_w)} N(V,\beta ) .\]
Alternatively, since cylinder spines only have one twig, this sum can be indexed by twig types. Given a curve class $\beta$, only finitely many twig types are realized by stable maps of class $\beta$.
\end{rmk}

\section{Primitive holomorphic cylinders} \label{sec-count-primitive-cyl}

In this section we prove the main Theorem \ref{theo:reduction-primitive-twig} together with its Corollary \ref{coro:reduction-elementary-cylinder}. The theorem is proved by using a deformation procedure parametrized by tropical curves. Using ideas similar to Kontsevich's formula for plane rational curves, we define a subspace in a moduli space of analytic stable maps, and obtain an equality between counts by looking at different degenerations of the domain curve in this subspace. The subspace is defined using tropical data. To prove that the counts are deformation invariants, we express them as the degree of a map which is proper and flat over the analytic deformation.

 In \S\ref{subsec:initial-data} we set up the notations for the tropical curves that will parametrize our deformation, and in \S\ref{subsec:tropical-deformation} we define the tropical deformation. In \S\ref{subsec:analytic-deformation} we pull back the tropical deformation to the analytic moduli space.
 The key results are Proposition \ref{prop:connected-components}, which shows that the tropical deformation cuts out connected components in the space of tropical curves, and Proposition \ref{prop:properness-restriction}, which shows that these connected components pull back to connected components of the smooth locus of the analytic moduli space up to restricting the domain curve. This properness argument relies on \ref{prop:nets-have-stable-domain} which is a small generalization of \linebreak\cite[][Proposition 10.1]{keelFrobeniusStructureTheorem2019a}.

 Finally, in \S\ref{subsec:degeneration} we look at different degenerations of the domain curve, and express the counts in terms of cylinder counts with one less twig together with some toric counts, which we evaluate explicitly. The key result is Proposition \ref{prop:splitting-formula}, relating counts before and after removing one twig to the tropical cylinder. Applying this formula inductively, we easily deduce Theorem \ref{theo:reduction-primitive-twig}.

\pagebreak
\subsection{Main results}
\begin{thm} \label{theo:reduction-primitive-twig}
Let $V$ be a primitive infinitesimal tropical cylinder with twig type $\w = (\w_s)_{1\leq s \leq t}$, let $\beta\in\NE (Y)$. Then
\[N(V, \beta ) = \sum_{\beta_1+\cdots +\beta_t = \beta} \prod_{s=1}^t N(V_s, \beta_s) ,\]
where $V_s$ is an infinitesimal cylinder of twig type $\w_s$ (see Construction \ref{constr:elementary-cylinders}).
\end{thm}

The next lemma allows to simplify the sum, by identifying the curve classes contributing to non zero invariants.

\begin{lemma}
  Let $V$ be a primitive infinitesimal cylinder of type $\p =(\p_1 , \p_2 )$, whose twig has a single leaf of type $\w = \p_1 + \p_2$ in the direction of $\Etrop_i$. Then, then are at most $\ell_i$ curve classes $\beta$ such that $N(V_s , \beta )\neq 0$.
\end{lemma}

\begin{proof}
Since the extension curve class is uniquely determined, the question\linebreak is equivalent to determining $\beta$ such that $\sm{M}_{0,3} (U,\p ,\beta)\cap \Trop^{-1} (\widehat{V}_w)$ is non empty.

A necessary condition for $\sm{M} (U,\p ,\beta)$ to be non empty is that the curve class ${\beta}$ be compatible with $\p$ \cite[][Remark 3.5]{keelFrobeniusStructureTheorem2019a}.
This determines the intersection numbers $\beta\cdot D_i$, so the only freedom in choosing $\beta$ lies in the intersection numbers $\beta\cdot E_{ij}$.

By construction, if a stable map $\eta$ in $\sm{M} (U ,\p ,\beta)$ meets the exceptional divisor $E_{kj}$ with multiplicty $m$, then $\Trop (\eta )$ has a twig in the direction $\Etrop_k$ of degree $m$. Thus, elements of $\sm{M} (U,\p ,\beta )\cap \Trop^{-1} (\widehat{V}_w)$ only meet a single irreducible component of $E_i$ with multiplicty $1$.
There are $\ell_i$ such components, each giving rise to a different curve class since the intersection numbers induce an isomorphism $N_1 (Y)\simeq N_1 (Y_t) \oplus \Z^E $.
\end{proof}

As a direct consequence, we get:

\begin{coro} \label{coro:reduction-elementary-cylinder}
  Let $V$ be a primitive infinitesimal cylinder of twig type\linebreak ${\w = (\w_s )_{1\leq s\leq t}}$. For each $s$, let $\Etrop_{i(s)}$ be the direction of the corresponding twig.

  Then there are at most $\prod_{1\leq s\leq t} \ell_{i(s)}$ curve classes such that $N(V,\beta )\neq 0$. Such a curve class $\beta$ is determined by the choice of an irreducible component $E_{i(s) j}$ for all $s$, and we then have
  \[N(V,\beta ) = \prod_{s=1}^t N(V_s ,\beta_s ),\]
  where $\beta_s$ is the curve class whose intersection number with each irreducible \linebreak component of $E$ is $0$ except for $E_{i(s) j}$.
\end{coro}

\subsection{Initial data and notations} \label{subsec:initial-data}

We fix once and for all a tropical cylinder $V = [\Gamma^0 , (v_1^0 ,v_2^0 ,v_w^0) , h^0]$ associated to a cylinder spine. Let $\widehat{V}$ denote the extended tropical cylinder associated to $V$. We assume that $V$ is primitive of\linebreak type ${\w = (\w_s )_{1\leq s\leq t}}$. This means that each $\w_s$ is a primitive integral vector, and that $V$ has a single twig with $t$ leaves. \\
Let $r$ and $(u_1,\dots , u_t)$ denote respectively the root and the leaves of the twig. We denote by $\w_0 = \sum_s \w_s$, and let $\sigma$ be the wall with direction $-\w_0$.

Our assumptions imply that the twig is a tree with a single $(t+1)$-vertex mapped to the origin. The $k$-th leaf is an interval $[0 ,+\infty]$, with $0$ mapped to the origin and $h^0$ being a bijection to the ray $\R_{\geq 0} \w_k$.
For each $k$, we fix $x_{t^k}\neq O$ and $x_{g^k}\in \intervalleoo{O}{x_{t^k}}$ in $\R_{\geq 0} \w_k$.
We also fix points $x_w$ and $x_{w'}$ not lying inside walls.

\begin{constr}\label{constr:elementary-cylinders}
For each $1\leq k\leq t$, we define an infinitesimal cylinder $V_k$ with twig type $\w_k$, \emph{i.e.} whose twig has a single leaf of degree $1$ in the $\w_k$ direction. We let the bending vertex be mapped to a point in $\intervalleoo{0}{x_g^k}$, and choose the contact orders of the two boundary legs such that $V_k$ is balanced.
\end{constr}

We now define the following tropical curves (Figure \ref{fig:tropical-curves}):

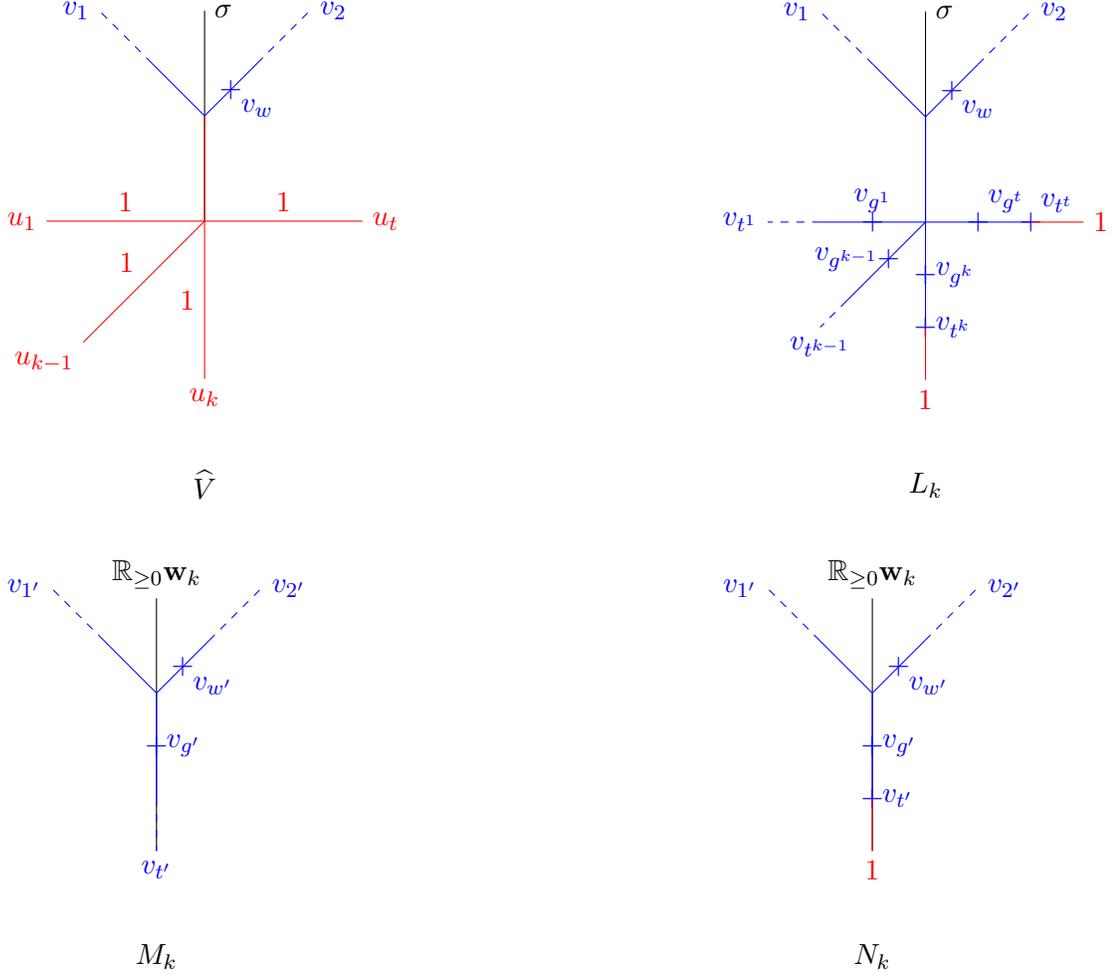
\begin{figure}
\begin{subfigure}{0.3\textwidth}
  \begin{tikzpicture}[scale=0.7]
\tikzstyle{length} = [<->, dotted, thin]
\tikzstyle{leg} = [dashed,blue]
\tikzstyle{twig} = [red]
\tikzstyle{interior-point} = [draw]
\tikzstyle{wall} = [thin]

  \draw[wall] (0,0) -- (0,4) node[right]{$\sigma$};

  \draw[blue] (0,2) -- (1,3) ;
  \draw[blue] (0,2) -- (-1,3) ;


  \draw[leg] (-1,3) -- (-2,4) node[above, left] {$v_1$};
  \draw[leg] (1,3) -- (2,4) node[above, right] {$v_2$} ;


  \node[blue] (w) at (0.5,2.5) {$+$};
  \draw[blue] (w.center) node[below right]{$v_w$};

  \draw[twig] (0,2) -- (0,0) ;
  \draw[twig] (0,0) -- (-3,0) node[midway , above]{$1$} node[left]{$ u_1$};
  \draw[twig] (0,0) -- (-2.3 , -2.3) node[midway , above left]{$1$} node[below left]{$u_{k-1}$};
  \draw[twig] (0,0) -- (0,-3)node[midway, left]{$1$} node[below] {$u_k $} ;
  \draw[twig] (0,0) -- (3,0) node[midway, above]{$1$} node[right]{$u_t$} ;

  \node (A) at (0,-5) {$\widehat{V}$} ;

\end{tikzpicture}
\end{subfigure}
\hfill
\begin{subfigure}{0.3\textwidth}
  \begin{tikzpicture}[scale=0.7]
\tikzstyle{length} = [<->, dotted, thin]
\tikzstyle{leg} = [dashed, blue]
\tikzstyle{twig} = [red]
\tikzstyle{interior-point} = [draw]
\tikzstyle{wall} = [thin]

  \draw[wall] (0,0) -- (0,4) node[right]{$\sigma$};

  \draw[blue] (0,0) -- (0,2) ;
  \draw[blue] (0,2) -- (1,3) ;
  \draw[blue] (0,2) -- (-1,3) ;
  \draw[blue] (0,0) -- (-2,0);
  \draw[blue] (0,0) -- (-1.5 , -1.5);
  \draw[blue] (0,0) -- (0,-2) ;
  \draw[blue] (0,0) -- (2,0) ;


  \draw[leg] (-1,3) -- (-2,4) node[above, left] {$v_1$};
  \draw[leg] (1,3) -- (2,4) node[above, right] {$v_2$} ;
  \draw[leg] (-2,0 ) -- (-3,0) node[left] {$v_{t^1}$};
  \draw[leg] (-1.5,-1.5) -- (-2 , -2) node[below]{$v_{t^{k-1}}$};


  \node[blue] (w) at (0.5,2.5) {$+$};
  \draw[blue] (w.center) node[below right]{$v_w$};

  \node[blue] (g) at (-1,0) {$+$} ;
  \draw[blue] (g.center) node[above]{$v_{g^1}$};

  \node[blue] (gg) at (-0.7,-0.7) {$+$} ;
  \draw[blue] (gg.center) node[left]{$v_{g^{k-1}}$};

  \node[blue] (gk) at (0,-1) {$+$};
  \draw[blue] (gk.center) node[right]{$v_{g^k}$};

  \node[blue] (tk) at (0,-2) {$+$};
  \draw[blue] (tk.center) node[right]{$v_{t^k}$};

  \node[blue] (gk) at (1,0) {$+$};
  \draw[blue] (gk.center) node[above right]{$v_{g^t}$};

  \node[blue] (tk) at (2,0) {$+$};
  \draw[blue] (tk.center) node[above right]{$v_{t^t}$};

  \draw[twig] (0,-2) -- (0,-3) node[below] {$1$};
  \draw[twig] (2,0) -- (3,0) node[right] {$1$};

  \node (A) at (0,-5) {$L_k$} ;

\end{tikzpicture}
\end{subfigure}
\\
\begin{subfigure}{0.3\textwidth}
  \begin{tikzpicture}[scale=0.7]
\tikzstyle{length} = [<->, dotted, thin]
\tikzstyle{leg} = [dashed,blue]
\tikzstyle{twig} = [red]
\tikzstyle{interior-point} = [draw,blue]
\tikzstyle{wall} = [thin]

  \draw[wall] (0,-1) -- (0,3.8) node[above]{$\R_{\geq 0} \w_k$};

  \draw[blue] (0,2) -- (1,3) ;
  \draw[blue] (0,2) -- (-1,3) ;
  \draw[blue] (0,0) -- (0,2) ;


  \draw[leg] (-1,3) -- (-2,4) node[above, left] {$v_{1'}$};
  \draw[leg] (1,3) -- (2,4) node[above, right] {$v_{2'}$} ;
  \draw[leg] (0,0) -- (0,-1) node[below] {$v_{t'}$};


  \node[blue] (w) at (0.5,2.5) {$+$};
  \draw[blue] (w.center) node[below right]{$v_{w'}$};

  \node[blue] (g) at (0,1) {$+$} ;
  \draw[blue] (g.center) node[right] {$v_{g'}$};


  \node (A) at (0,-3) {$M_k$} ;

\end{tikzpicture}
\end{subfigure}
\hfill
\begin{subfigure}{0.3\textwidth}
  \begin{tikzpicture}[scale=0.7]
\tikzstyle{length} = [<->, dotted, thin]
\tikzstyle{leg} = [dashed,blue]
\tikzstyle{twig} = [red]
\tikzstyle{interior-point} = [draw]
\tikzstyle{wall} = [thin]

  \draw[wall] (0,-1) -- (0,3.8) node[above]{$\R_{\geq 0} \w_k$};

  \draw[blue] (0,2) -- (1,3) ;
  \draw[blue] (0,2) -- (-1,3) ;
  \draw[blue] (0,0) -- (0,2) ;


  \draw[leg] (-1,3) -- (-2,4) node[above, left] {$v_{1'}$};
  \draw[leg] (1,3) -- (2,4) node[above, right] {$v_{2'}$} ;


  \node[blue] (w) at (0.5,2.5) {$+$};
  \draw[blue] (w.center) node[below right]{$v_{w'}$};

  \node[blue] (g) at (0,1) {$+$} ;
  \draw[blue] (g.center) node[right] {$v_{g'}$};

  \node[blue] (t) at (0,0) {$+$} ;
  \draw[blue] (t.center) node[right] {$v_{t'}$};

  \draw[twig] (0,0) -- (0,-1) node[below] {$1$} ;

  \node (A) at (0,-3) {$N_k$} ;

\end{tikzpicture}
\end{subfigure}

  \caption{The initial tropical cylinder $\widehat{V}$, the and the families of tropical curves $L_k$, $M_k$ and $N_k$.}
  \label{fig:tropical-curves}
\end{figure}

\begin{enumerate}
  \item $L_1$ is obtained from $\widehat{V}$ by adding two interior points attached to every direction of the twig. More precisely for the $i$-th leaf, we denote by $\overline{v}_{g^i}$ (resp. $\overline{v}_{t^i}$) the unique point being mapped to $x_{g^i}$ (resp. $x_{t^i}$), and glue to it the constant leg $[0 , +\infty = v_{g^i}]$ (resp. $[0 ,+\infty = v_{t^i}]$).
  \item For $2\leq k\leq t+1$, $L_k$ is obtained from $L_{k-1}$ by forgetting the $k$-th leaf of the twig, and assigning to the map the weight $w_k$ along the $t^k$-leg.
  \item $M_k$ is the balanced spine obtained from the extended tropical cylinder $\widehat{V}_k$ by substituting the leaf of the twig with a boundary marked point $v_{t'}^M$, and adding an interior leg $v_{g'}^M$ mapped to $x_{g^k}$ to this new leg.
  \item $N_k$ is obtained from the extended tropical cylinder $\widehat{V}_k$ by adding two interior legs attached to the twig, the new marked point $v_{t'}^N$ (resp. $v_{g'}^N$) being mapped to $x_{t^k}$ (resp. $x_{g^k}$).
\end{enumerate}

The curve $L_k$ has $2t+3$ marked points, indexed by the set
\[{J^L = \lrbrace{1,2,w , g^1, t^1, \dots ,g^t ,t^t}}. \]
 Let $B_0^L =\lrbrace{1,2}$ and $I_0^L = J^L\setminus B_0^L$ denote the boundary and interior marked points of $L_0$. For $1\leq k \leq t+1$, we set $B_k^L = B_{k-1}^L \cup \lrbrace{t^{k-1}}$ and $I_k^L = I_{k-1}^L \setminus \lrbrace{t^{k-1}}$.
These sets index the boundary and interior marked points of $L_k$.

The tropical curves $M_k$ and $N_k$ have $5$ marked points, indexed by the set ${J^M = J^N = \lrbrace{1',2',w',g',t'}}$. For the $M$-sequence the interior points are indexed by $I^M = \lrbrace{w',g'}$ and the boundary points by $B^M = \lrbrace{1',2',t'}$, while for the $N$-sequence the interior points are indexed by $I^N = \lrbrace{w',g',t'}$ and the boundary points by $B^N = \lrbrace{1',2'}$.

Finally, we introduce
\[J^g = \lrbrace{1,2,1',2' , w,w'  , g^1 , t^1\dots  , g^t , t^t, t'} = J^L\cup J^M\setminus\lrbrace{g'}, \]
the set $J^g$ has cardinality $2t+7$. For $0\leq k\leq t$ we define the partition given by boundary indices
\[B_k^g = \lrbrace{1,2,1',2' , t^1 ,\dots , t^{k-1} , t'} = B_k^L\cup B^M ,\]
 and interior indices
 \[I_k^g = \lrbrace{w,w' , g^1 , \dots, g^{k-1} , g^k , t^k , \dots , g^t ,t^t} = I_k^L\cup I_k^M\setminus \lrbrace{g'} .\]

Given a set $J$ indexing marked points and a subset of interior indices $\tilde{I}\subset J$,\linebreak we denote by $\ev_{\tilde{I}}$ the map given by simultaneous evaluation at marked points\linebreak of $\widetilde{I}$.

For $i\in\lrbrace{L,M,N,g}$ and $1\leq k\leq t$, denote by $\p_k^i$ a tuple of weights of length $n_i\coloneqq \lrvert{J^i}$. At the level of analytic moduli spaces, we define the following maps:
\begin{enumerate}
  \item $\Phi_k^L = (\st , \ev_{I_k^L})\colon \overline{{M}}(\an{Y} , \p_k^L ,\beta )\rightarrow \an{\overline{{M}}}_{0,2t+3}\times \lr{\an{Y}}^{2t-k+2}$.
  \item $\Phi_k^M = (\st , \ev_{I^M})\colon \overline{{M}} (\an{Y} , \p_k^M ,\beta )\rightarrow \an{\overline{{M}}}_{0,5}\times \lr{\an{Y}}^{2}$.
  \item $\Phi_k^N = (\st , \ev_{I^N})\colon \overline{{M}} (\an{Y} , \p_k^N ,\beta )\rightarrow \an{\overline{{M}}}_{0,5}\times \lr{\an{Y}}^{3}$.
  \item $\Phi_k^g = (\st , \ev_{I_k^g})\colon \overline{{M}} (\an{Y} , \p_k^g ,\beta )\rightarrow \an{\overline{{M}}}_{0,2t+7}\times \lr{\an{Y}}^{2t-k+3}$.
\end{enumerate}
 These maps tropicalize to maps between the tropicalization of spaces, which we denote by
 \[\sf{\Phi}_k^i\colon \TC (\base , \p_k^i)\rightarrow \overline{\mathsf{M}}_{0,n_i}\times \base^{I_k^i},\]
 where $\overline{\mathsf{M}}_{0,n_i}$ denotes the moduli space of pointed tropical curves \cite{abramovich_tropicalization_2014}.

\subsection{Tropical deformation} \label{subsec:tropical-deformation}

\begin{constr} \label{constr:trop-divisor}
For $1\leq k\leq t$, fix a primitive vector $\w_k'$ such that the mixed volume of $(\w_k, \w_k')$ is equal to $1$. We consider the line $\sf{H}_k^g$ with direction $\w_k'$ going through $x^{g^k}$, and the line $\sf{H}_k^t$ with direction $\w_k'$ going through $x_k^t$.
Their equations are given by integral affine functions on $\base$, that we pull back to a equations defining Cartier divisors $H_k^g$ and $H_k^t$ on $Y$. By the tropical intersection formula \cite{katz_tropical_2012}, we have $H_k^g\cdot D_{k} = H_k^t\cdot D_k = 1$.
\end{constr}

\begin{constr}
For $1\leq k\leq t+1$, we define:
\begin{enumerate}
  \item $\sf{V}_k^L = \sf{st} (L_k) \times x_w\times x_{g^1} \times \cdots x_{g^{k-1}} \times \sf{H}_k^g \times \sf{H}_k^t\times \cdots \times \sf{H}_k^t\times \sf{H}_k^g$.
  \item $\sf{V}_k^M = \sf{st} (M_k) \times x_{w'} \times x_{g^k}$.
  \item $\sf{V}_k^N = \sf{st}  (N_k) \times x_{w'}\times \sf{H}_k^g\times \sf{H}_k^t$.
\end{enumerate}
and set $\sf{TC}_k^i  =\lr{\sf{\Phi}_k^i}^{-1} (\sf{V}_k^i)$ for $i\in\lrbrace{L,M,N}$.
\end{constr}

\begin{prop} \label{prop:isolated-point}
  The point $L_k\in \sf{TC}_k^L$ is isolated. Similarly, $M_k\in\sf{TC}_k^M$ and $N_k\in\sf{TC}_k^N$ are isolated.
\end{prop}

\begin{proof}
  We need to prove that $L_k$ does not deform in $\sf{TC}_k^L$. Let $U = U(L_k ,\varepsilon , \ca{W})$ be an elementary open neighbourhood of $L_k$ in $\TC (\base , \p_k^L)$.

  We first note that the only possible deformations of the domain $\Gamma_k^L$ of $L_k$ in $\TC (\base ,\p_k^L )$ consist in changing the length of the finite edges (in particular, moving around the roots of the twigs), and deforming the unique $(t+1)$-vertex into lower valency vertices. Up to choosing the cover $\ca{W}$ such that twigs of $L_k$ are mapped to disjoint regions of $\bbase$, we can assume that twigs remain intervals throughout every deformation in $U$, and that their images do not contain the origin.

  Now let $K = [\Gamma , (v_j)_{j\in J_k^L} , h]\in U\cap \sf{TC}_k^L$, by definition we have a contraction $c\colon \Gamma\rightarrow \Gamma_k^L$. First we claim that the condition on the domain of $K$ ensures that the $(t+1)$-valent vertex does not break into several vertices of lower valency. Otherwise, since every edge incident to this vertex lies in the spine, these multiple vertices would show in $\sf{st} (K)$, which would then not equal $\sf{st} (L_k)$. Hence the domains of $K$ and $L_k$ coincide as combinatorial trees (without metric structure). The weight of $h$ on every edge is fixed, so $K$ is completely determined by the length of finite edges and the image of a single point. Let $v_0$ the $(t+1)$-valent vertex in $K$.

  For $i\geq k$, the root of the $i$-th twig in $L_k$ is a $3$-valent vertex. Thus it does not deform, and comes from a $3$-valent vertex $r_i'$ in $\Gamma$. Since intervals with a one-valent infinite vertex do not deform either, two of the edges incident to $r_i'$ are fixed: one is the $i$-th twig, the other is the $t^i$-leg. Then $r_i'$ lies in $\sf{H}_i^t\cap \R_{\geq 0} w_i = \lrbrace{x_{t^i}}$, hence the root of the twig does not deform. The condition on $v_{g^i}$ fixes the length of the two finite edges making up the path from $v_0$ to $r_i'$.

  For $i<k$, the length of the edge connecting $v_0$ to the $g^i$-leg is fixed by the condition that $v_{g^i}$ maps to $x_{g^i}$. Similarly, the condition on the image of $v_w$ fixes the lengths of the edges in the path connecting $v_w$ to $v_0$. Finally, we note that $v_0$ is mapped to the origin, since there are multiple edges incident to $v_0$ that are mapped to distinct walls, so the length of ever finite edge in $\Gamma$ equals the length of the corresponding vertices in $\Gamma_k^L$.

  All these observations put together show that $K = L_k$, proving the claim. The proof is similar, but simpler, for $M_k$ and $N_k$.
\end{proof}

We let $\sf{T}_k^L = \lrbrace{L_k}$, $\sf{T}_k^M = \lrbrace{M_k}$ and $\sf{T}_k^N = \lrbrace{N_k}$. The proposition is saying that $\sf{T}_k^i\subset \sf{TC}_k^i$ is a connected component.

\begin{constr}
Let $r\in\intervalleff{0}{+\infty}$. For $1\leq k \leq t$, consider the element $\Gamma_{k,r}$ (resp. $\Gamma_{k,r'}$) in $\sf{\overline{M}}_{0,2t+7}$ as in figure \ref{fig:trop-path},
obtained by gluing the stabilization of domains of $L_k$ and $M_k$ (resp. $L_{k+1}$ and $N_k$) along the vertices $v_{g^k}$ and $v_{g'}$, and varying the length of the horizontal edge, equal to the parameter $r$.
When $r=0$ the two abstract graphs are the same, so we obtain a path $\Delta\subset\sf{\overline{M}}_{0,n+7}$ parametrized by $r\in\intervalleff{-\infty}{+\infty}$, whose marked points indexed by $J^g$ and are partitioned into interior and boundary marked points as $J^g = I_k^g \cup B_k^g$.

\begin{figure}
  \begin{tikzpicture}
\tikzstyle{length} = [<->, dotted, thin]
\tikzstyle{twig} = [red]

\begin{scope}[scale=0.6]
\node[draw,circle] (L) at (-3.5,0) {$L_k$} ;
\node[draw,circle] (M) at (3.5,0) {$M_k$};
\draw (L.east)  -- (M.west) ;
\draw (0,0) -- (0,-2) ;
\draw[dashed] (0,-2) -- (0,-3)  node[below] {$v_g$};
\draw[length] (-2.7 ,0.2) -- (2.6,0.2) node[midway, rotate=0, above]{$r$};

\draw (L.south) -- (-3.5 , -2) -- (-4.5,-2);
\draw[dashed] (-4.5,-2) -- (-5.5,-2) node[left] {$v_{t^k}$ };

\draw (M.south) -- (3.5 , -2);
\draw[dashed] (3.5,-2) -- (3.5,-3) node[below] {$v_{t'}$ };

\draw[twig] (-3.5 , -2) -- (-3.5 , -3) node[below] {$u_k$};
  \node[] (1) at (0,-4.5) {$\widehat{\Gamma_r'}$};
\end{scope}

\begin{scope}[scale=0.6, shift={(10,0)}]
  \node[draw,circle] (L) at (-3.5,0) {$L_{k+1}$} ;
  \node[draw,circle] (M) at (3.5,0) {$N_k$};
  \draw (L.east)  -- (M.west) ;
  \draw (0,0) -- (0,-2) ;
  \draw[dashed] (0,-2) -- (0,-3)  node[below] {$v_g$};
  \draw[length] (-2.4 ,0.2) -- (2.7,0.2) node[midway, rotate=0, above]{$r$};

  \draw (M.south) -- (3.5 , -2) -- (4.5,-2);
  \draw[dashed] (4.5,-2) -- (5.5,-2) node[right] {$v_{t^k}$ };

  \draw (L.south) -- (-3.5 , -2);
  \draw[dashed] (-3.5,-2) -- (-3.5,-3) node[below] {$v_{t'}$ };

  \draw[twig] (3.5 , -2) -- (3.5 , -3) node[below] {$u_k$};

  \node[] (A) at (0,-4.5) {$\widehat{\Gamma_r}$};

\end{scope}

\end{tikzpicture}
  \caption{The path of tropical curves in $\sf{\overline{M}}_{0,2t+7}$, we only swap the $t^k$-leg together with the $k$-th twig with the $t'$-leg. The other twigs ($k+1$ through $t$) of $L_k$ are not modified.}    \label{fig:trop-path}
\end{figure}
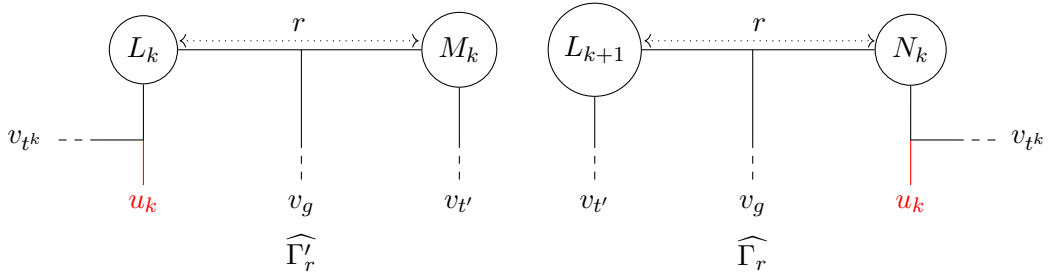
\end{constr}

\begin{constr}
  For $1\leq k\leq t$, we set:
  \[\sf{V}_k^g = \Delta \times x_{w} \times x_{w'} \times x_{g^1}\times \cdots \times x_{g^{k-1}}\times \sf{H}_k^g\times \sf{H}_k^t\times\cdots\times \sf{H}_t^g\times \sf{H}_t^t ,\]
  and define $\sf{TC}_k^g \coloneqq \lr{\sf{\Phi}_k^g}^{-1} (\sf{V}_k^g)$.
\end{constr}

We will work with tropical curves in $\TC_k^g$.
However in $\TC_k^g$ there are many tropical curves which are irrelevant to us, since we only impose conditions on the domain tropical curve and on interior points, but not on the twigs. In the next construction, we select the connected components in $\TC_k^g$ containing the relevant tropical curves for our count.

\begin{constr} \label{constr:connected-component}
For $K=[\Gamma , (v_j) , h]\in \TC_k^g$, we denote by $P_{ij}$ the path in $\Gamma$ from that vertex $v_i$ to the vertex $v_j$ in the spine of $K$.\\
Let $\sf{T}_k^g$ be the subset of $\TC^g$ consisting of tropical curves $K=[\Gamma , (v_j) , h]\in \TC^g$ such that:
\begin{enumerate}
  \item To $P_{wt^k}\cap P_{w't^k}$ is attached a single twig of degree $1$, with direction $\w_k$.
  \item For $k+1\leq i\leq t$, to $P_{wt^i}\cap P_{g^it^i}$ is attached a single twig of degree $1$, with direction $\w_i$.
\end{enumerate}
\end{constr}

\begin{prop} \label{prop:connected-components}
For $1\leq k\leq t$, the subset $\sf{T}_k^g\subset \TC_k^g$ is a union of connected components.
\end{prop}

\begin{proof}
  We first prove that $\sf{T}_k^g$ is open. Let $K = [\Gamma , (v_j) , h]\in \sf{T}_k^g$. Let $\varepsilon >0$, let $\ca{W}$ be a finite open cover of $\bbase$ and consider $U = U(K,\varepsilon , \ca{W})\cap \TC_k^{g}$. By definition, if $K' = [\Gamma ' ,(v_j') , h']\in U$ then we have a continuous map $c\colon \Gamma ' \rightarrow \Gamma$ contracting a subset of topological edges, sending $v_j'$ to $v_j$ and nodes to nodes.\\
  Let $P_k = P_{wt^k}\cap P_{w't^k}$ in the domain of $\Gamma'$, and for $i\geq k+1$ let $P_i = P_{wt^i}\cap P_{g^it^i}$ in $\Gamma '$. For $i\geq k$, the root $r_i$ of the $i$-th twig does not deform since it is a $3$-valent vertex. Thus $r_i' = c^{-1} (r_i )$ is still a $3$-valent vertex in $\Gamma '$. Furthermore, as intervals with an infinite $1$-valent vertex do not deform either, two of the edges incident to $r_i'$ are fixed. One of them is the $t^i$-leg and the other is the $i$-th twig.
  In particular $r_i'$ is the root of the $i$-th twig in $K'$, and is also the endpoint of the $t^i$-leg. Thus $r_i'\in P_i$, proving conditions (1) and (2) of Construction \ref{constr:connected-component}.

  We now prove $\sf{T}_k^g$ is closed.
  Let $(K_{\lambda} )_{\lambda\in\Lambda}$ be a net in $\sf{T}_k^g$ converging to \linebreak ${K_{\infty} = [\Gamma_{\infty} , (v_j^{\infty} ) ,h_{\infty}]}$ in $\TC_k^g$. Let $U$ be a basic neighbourhood of $K_{\infty}$ in $\TC_k^g$ as in the proof of openness. By definition, there exists $\lambda_0\in\Lambda$ such that \linebreak${K_{\lambda_0} = [\Gamma_{\lambda_0} , (v_j^{\lambda_0}) , h_{\lambda_0}]\in U}$.
  Thus we have a continuous map $c\colon\Gamma_{\lambda_0}\rightarrow \Gamma$ \linebreak contracting a subset of edges. Conditions (1) and (2) of Construction \ref{constr:connected-component} are still satisfied after contraction of some edges, thus $K_{\infty}\in \sf{T}_k^g$.
\end{proof}

\subsection{Analytic deformation} \label{subsec:analytic-deformation}

We proceed to defining the relevant spaces of \linebreak analytic curves lying above the tropical spaces. To do this, recall the commutative diagram with tropicalization maps
\begin{cd}
  \sm{{M}} (U,\p_k^i ,\beta ) \ar[d,"\Trop^i"] \ar[r,hook] & \overline{{M}} (Y,\p_k^i ,\beta )\ar[r, "\Phi_k^i"] & \overline{{M}}_{0,n_i} \times Y^{I_k^i} \ar[d, "\Trop^i"] \\
  \TC (\base , \p_k^i) \ar[rr, "\mathsf{\Phi}_k^i"] & & \overline{\mathsf{M}}_{0,n_i} \times \overline{M}_{\R}^{I_k^i}
\end{cd}
At the level of domain curves, the tropicalization map corresponds to taking the convex hull of the marked points. For stable maps, the tropicalization map gives a well defined map to $\TC (\base ,\p_k^i )$ on the smooth locus only due to our notion of tropical curves.

\begin{constr}
Given a substack ${M}\subset \overline{{M}} (Y, \p_k^i ,\beta)$ we denote by $\sd{{M}}$ (resp. $\sm{{M}}$) its restriction to $\sd{{M}} (U , \p_k^i ,\beta ) $ (resp. to $\sm{{M}} (U ,\p_k^i ,\beta)$).\linebreak
For $i\in\lrbrace{L,M,N,g}$ and $1\leq k \leq t+1$, define the following substacks :
\begin{enumerate}
  \item $V_k^i \coloneqq \lr{\Trop^i}^{-1} (\sf{V}_k^i)$ in $\overline{{M}}_{0,n_i}\times Y^{I_k^i}$.
  \item $\overline{{M}} (V_k^i ,\beta ) =\lr{\Phi_k^i }^{-1} \lr{V_k^i}$ in $\overline{{M}} (Y,\p_k^i ,\beta)$.
  \item ${M} (T_k^i , \beta )= \lr{\Trop^i}^{-1} (\sf{T}_k^i)\cap\sm{{M}} (U,\p_k^i ,\beta )$.
\end{enumerate}
By construction, we have ${M} (T_k^i  ,\beta )\subset\sm{\overline{{M}} (V_k^i ,\beta )}$. We continue to denote by $\Phi_k^i\colon{M}(T_k^i , \beta )\rightarrow V_k^i$ the restriction.
\end{constr}

We have natural maps induced by forgetting marked points
\[V_k^g \longrightarrow V_k^i ,\; i\in\lrbrace{L,M,N}, \text{ and } V_k^g\longrightarrow V_{k+1}^L .\]
These maps are proper and flat, in particular open, since they are given by forgetting marked points and projections $Y^{I_k^g}\rightarrow Y^{I_k^i}$.

The following proposition expresses that $M(T_k^i ,\beta )$ is not too far from being compact.

\begin{prop}  \label{prop:nets-have-stable-domain}
  Let $(f_{\lambda})_{\lambda\in\Lambda}$ be a net in ${M} (T_k^i ,\beta)$, such that $(\Trop^i (f_{\lambda}))_{\lambda\in\Lambda}$ converges in $\TC_k^i$. Then a subnet of $(f_{\lambda})_{\lambda\in\Lambda}$ converges in $\sd{\overline{{M}} (V_k^i ,\beta)}$.
\end{prop}

\begin{proof}
  By properness of $\overline{{M}} (Y,\p_k^i ,\beta )$, up to passing to a subnet we may assume that $(f_{\lambda})$ converges to some $f_{\infty} = [C_{\infty} , (p_{j,\infty})_{j\in J} , C_{\infty}\rightarrow Y]\in\overline{{M}} (Y,\p_k^i ,\beta )$.

  We proceed as in \cite[][Proposition 10.1]{keelFrobeniusStructureTheorem2019a} by cutting the domain curves into body and caps. This decomposition is obtained by choosing a trivial family of closed disks centered at each boundary point. For $\lambda\in\Lambda\cup\lrbrace{\infty}$, denote by $\D_{i,\lambda}$ the (closed) cap associated to the $i$-th boundary marked point of $f_{\lambda}$ and by $\B_{\lambda}$ the corresponding body.

  The proof goes in three steps:
  \begin{enumerate}
    \item The boundaries of caps are mapped to a compact subset inside the torus \cite[Claim 10.3]{keelFrobeniusStructureTheorem2019a}.
    \item $f_{\infty} (\B_{\infty})\cap D = \emptyset$ \cite[Claim 10.4]{keelFrobeniusStructureTheorem2019a}.
    \item The limit caps $\D_{\infty , i}$ do not have bubbles \cite[Claim 10.5]{keelFrobeniusStructureTheorem2019a}.
  \end{enumerate}

  The first two claims carry on to our situation, but the proof of the third claim fails because of the non-transverse at infinity part of the spines we are considering. In the surface case, we can still prove that $\D_{i,\lambda}$ has no bubbles. Let $i\in B$ correspond to a non transverse at infinity boundary point. For $\lambda\in\Lambda\cup\lrbrace{\infty}$, let $v_{i,\lambda}$ be the image of the $i$-th marked point on the domain curve of the tropicalization, and let $b_{i,\lambda}$ denote the image of $\partial\D_{i,\lambda}$. Let $V$ be a compact polyhedral subset containing $h_{\lambda} (\intervalleff{b_{i, \lambda}}{v_{i,\lambda }})$ for all $\lambda\in\Lambda$, so that it also contains $h_{\infty} (\intervalleff{b_{i,\infty}}{v_{i,\infty}})$.
  Let $\widetilde{V} = (\pi\circ\tau_t)^{-1} (V)$.
  We can shrink $V$ so that $\overline{\widetilde{V}\setminus E}$ is affinoid and only meets the irreducible component $D_i$ of $D$. In addition to this affinoid domain, $\widetilde{V}$ contains a union of irreducible components of the exceptional divisor $E$.
  As $f_{\lambda} (\D_{i,\lambda })\subset\widetilde{V}$ for all $\lambda\neq\infty$, this inclusion also holds for $\lambda= \infty$ by compactness of $\widetilde{V}$ and continuity of the universal stable map.

  Let $C_i$ denote the unique irreducible component in $C_{\infty}$ intersecting both $\B_{\infty}$ and $\D_{i,\infty}$, which satisfies $f_{\infty, \ast} [C_i]\cdot D_i >0$. Let $C_b$ be a connected component of $\overline{\D_{i,\infty}\setminus C_i}$, it is a tree of $\mathbb{P}^1$. We note that each irreducible component $C'\subset C_b$ not contracted by $f_{\infty}$ has image equal to an irreducible component of $E$,
  thus contributes by $f_{\infty,\ast} [C']\cdot D = \deg f_{\infty\vert C'} >0$ to the intersection number $\beta\cdot D_i$. If $C'\subset C_b$ is such a component, then we have
  \[1= \beta\cdot D_i \geq f_{\infty,\ast} [C_i]\cdot D_i + \deg f_{\infty\vert C'} \geq 2  .\]
  This is a contradiction, so every component in $C_b$ is contracted to a point. In turn, this contradicts the stability of $f_{\infty}$ so we must have $C_b = \emptyset$. Then $f_{\infty}$ has stable domain.

\end{proof}

We will use the next lemma to reduce enumerative computations to the smooth part of the moduli space. Defining the invariants in this way allows to interprete simple invariants as naive counts. The idea is that if families in a subspace $\sm{M} (U,\p ,\beta )$ degenerate to at worst stable maps with stable domains, then we can obtain a closed subspace in $\sm{M} (U,\p ,\beta)$ by removing stable maps with domain curves arising as degenerations.

\begin{lemma} \label{lemma:restr-proper-connected-components}
  Let $V\subset \overline{{M}}_{0,n}\times Y^I$, let ${M}_V = \Phi^{-1} (V)\subset \overline{{M}} (Y,\p ,\beta )$, where $\Phi = (\st, \ev_I)$.
  Let ${M} '\subset {M}_V$, and assume that:
  \begin{itemize}
    \item ${M}'$ is a union of connected components in $\sm{{M}}_V$.
    \item ${M}'$ has Zariski closure in ${M}_V$ contained in $\sd{{M}}_V$.
  \end{itemize}
  Then there exists a Zariski open $W\subset V$, such that ${M}_W\coloneqq \Phi^{-1} (W)$ satisfies:
  \begin{enumerate}
    \item ${M}_W'\coloneqq {M}'\cap {M}_W$ is a union of connected components in ${M}_W$. \linebreak In particular ${M}_W'\subset\sm{{M}}_V$.
    \item The restriction $\Phi\colon {M}_{W}'\rightarrow W$ is proper.
    \item $W$ intersects every fiber of the first projection.
  \end{enumerate}
\end{lemma}

\begin{proof}
  Let $W = V\setminus \Phi (\overline{{M}'}\setminus \sm{{M}_V})$, it is Zariski open by properness of $\Phi$. Let ${M}_W\coloneqq \Phi^{-1} (W)$.

  By construction $\overline{{M}'}\cap {M}_W \subset \sm{{M}_{V}}$, so the Zariski closure of ${M}_W'$ in ${M}_{W}$ lies in $\sm{{M}_{W}}$.
  Since ${M}'\subset \sm{{M}_V}$ is a union of connected components, ${M}_{W}' $ is a union of connected components in $\sm{{M}_{W}}$. In particular it equals its Zariski closure in ${M}_W$, thus ${M}_{W}'\subset {M}_W$ is a union of connected components proving (1).

  Since $\Phi\colon {M}_V\rightarrow V$ is proper, we deduce the properness of the restriction of $\Phi$ to ${M}_{W}'$ by base change and restriction to a union of connected components, proving (2).

  (3) follows from \cite[][Lemma 3.12]{keelFrobeniusStructureTheorem2019a}, the lemma states that fibers over a fixed domain curve have dense images under evaluation maps.
\end{proof}

We apply the previous lemma to the subspaces ${M} (T_k^i , \beta )$. For the degeneration argument to work, we impose a further compatibility condition on our choice of Zariski open susets.

\begin{prop} \label{prop:properness-restriction}
  There exist Zariski opens $W_k^i\subset V_k^i$ such that:
  \begin{enumerate}
    \item ${M} (T_k^i , \beta )_{W_k^i}$ is contained in $\sm{\overline{{M}} (V_k^i ,\beta )}$.
    \item The restriction $\Phi_k^i\colon {M} (T_k^i,\beta )_{W_k^i}\rightarrow W_k^i$ is proper.
    \item $W_k^i$ intersects every fiber of the first projection map.
    \item The forgetful maps $W_k^g\rightarrow W_k^i$ and $W_k^g\rightarrow W_{k+1}^L$ are surjective.
  \end{enumerate}

\end{prop}

\begin{proof}
By Propositions \ref{prop:connected-components} and \ref{prop:nets-have-stable-domain},
we can apply Lemma \ref{lemma:restr-proper-connected-components} with $V = V_k^i$ and ${M}' = {M} (T_k^i , \beta )$. This gives Zariski opens $\widetilde{W}_k^g$, $\widetilde{W}_k^i$ and $\widetilde{W}_{k+1}^L$ that satisfy (1), (2) and (3).

We then let
\[W_k^g = \widetilde{W}_k^g \cap \Fgt^{-1}\widetilde{W}_k^L \cap\Fgt^{-1}\widetilde{W}_k^M \cap \Fgt^{-1}\widetilde{W}_k^N \cap \Fgt^{-1}\widetilde{W}_{k+1}^L .\]
This is an open subset, and noting that the forgetful maps are proper and flat thus open, we define the open subsets
\[W_k^i = \Fgt (W_k^g)\subset \widetilde{W}_k^i \text{ and } W_{k+1}^L = \Fgt (W_k^g) \subset \widetilde{W}_{k+1}^L .\]
Conditions (1) and (2) holds by pullback, condition (3) holds by \cite[][Lemma 3.12]{keelFrobeniusStructureTheorem2019a} and (4) holds by construction.
\end{proof}
We denote by ${M} (T_k^i , \beta )_W$ the fiber product ${M} (T_k^i , \beta )\times_{V_k^i} W_k^i$. The following corollary is precisely what we need to define deformation invariant counts in the next subsection.

\begin{coro} \label{coro:proper-and-flat}
  Let:
  \begin{enumerate}
    \item $\Fgt_k^L\colon \overline{{M}}_{0,2t+3}\times Y^{2t-k+2}\rightarrow Y^{2t-k+2}$ denote the map forgetting every marked point except $w,g^k,t^k$.
    \item $\Fgt_k^M\colon \overline{{M}}_{0,5}\times Y^2\rightarrow Y^2$ denote the map forgetting every marked point except $w',g',t'$.
    \item $\Fgt_k^N\colon \overline{{M}}_{0,5}\times Y^{3}\rightarrow Y^{3}$ denote the map forgetting every marked point except $w', g', t'$.
    \item $\Fgt_k^g\colon \overline{{M}}_{0,2t+7}\times Y^{2t-k+3}\rightarrow \overline{\ca{M}}_{0,5}\times Y^{2t-k+3}$ denote the map forgetting every marked point except $w,w',g^k,t^k,t'$.
  \end{enumerate}
  The composition $\Psi_k^i \coloneqq \Fgt_k^i \circ \Phi_k^i$ restricted to ${M} (T_k^i ,\beta )_{W_k^i}$ is proper and flat.
\end{coro}

\begin{proof}
  By Proposition \ref{prop:properness-restriction}, the map $\Phi_k^i$ is proper. It is also smooth by \cite[][Lemma 3.6]{keelFrobeniusStructureTheorem2019a}. Forgetful maps between moduli space of stable curves correspond to universal families. Thus they are proper and flat, so $\Fgt_k^i$ is proper and flat. We deduce that the composition $\Psi_k^i$ is proper and flat.
\end{proof}

\subsection{Enumerative invariants and degeneration} \label{subsec:degeneration}

We now define enumerative invariants associated to the various spaces constucted. The definitions circumvents fundamental class because we managed to restrict to the smooth locus of the moduli spaces.

Recall that the dimension of the moduli space $\sm{{M}} (U ,\p ,\beta)$ for full-tangency $n$-pointed stable maps is $n-3 + \dim Y = n-1$ for the surface case. We denote by $q^i\colon {M} (T_k^i ,\beta )\rightarrow \pt$ the structure morphism, which is proper.

Given a (derived) $k$-analytic space over a point $q\colon X\rightarrow \pt$, the \emph{motivic \linebreak cohomology groups} $H^{2r} (X,\Q (r ))$ are defined as the Borel-Moore homology of the identity morphism ${\id_X\colon X\rightarrow X}$. Throughout, when $q$ is proper derived lci and \linebreak $\gamma\in H^{\ast} (X ,\Q (\ast ) )$ we use the virtual fundamental class $[X]$ to define
\[\int_{X} \gamma \coloneqq q_{\ast} (\gamma\cap [X])\in\Q .\]

Denote by $\pt\in H^{4} (Y, \Q (2) )$ and by $\pt_{\mu}\in H^{2(n-3)} (\overline{{M}}_{0,n} ,\Q (n-3))$ point classes, and let ${\delta_s^{\ell} = [H_s^{\ell}]\in H^{2} (Y,\Q (1)) }$ for $1\leq s\leq t$ and $\ell\in \lrbrace{g,t}$.

\begin{dfn} \label{def:glued-invariant}
  Let $\beta\in \NE (Y)$, we define:
  \begin{enumerate}
    \item $N (T_k^g  ,\beta ) \coloneqq \int_{{M} (T_k^g , \beta)_{W}} \lr{\Psi_k^{g}}^{\ast} \lr{\pt_{\mu} \extprod_{I_k^g\setminus\lrbrace{g^s,t^s}_{s\geq k}} \pt
      \extprod_{k\leq s \leq t} (\delta_s^{g}\extprod \delta_s^t)}$.
    \item $N(L_k ,\beta ) = \int_{{M} (T_k^L, \beta )_{W}} \lr{\Psi_k^L}^{\ast}\lr{\extprod_{i\in I_k^l\setminus \lrbrace{g^s, t^s}_{s\geq k}} \pt \extprod_{k\leq s\leq t} (\delta_s^g\extprod \delta_s^t)}$.
    \item $N(M_k ,\beta ) = \int_{{M} (T_k^M, \beta )_W} \lr{\Psi_k^M}^{\ast} (\pt\extprod\pt)$.
    \item $N(N_k ,\beta ) = \int_{{M} (T_k^N, \beta )_W} \lr{\Psi_k^N}^{\ast} (\pt\extprod\delta_k^g\extprod\delta_k^t)$.
  \end{enumerate}
\end{dfn}

We relate these invariants by computing $N(T_k^g , \beta)$ at different degenerations of the domain curve (Figure \ref{fig:degeneration-domain-curve}). To characterize these degenerations, we only need to remember the shape of the domain curve. This is why we introduced the forgetful maps and defined our invariants through the maps $\Psi_k^i$.

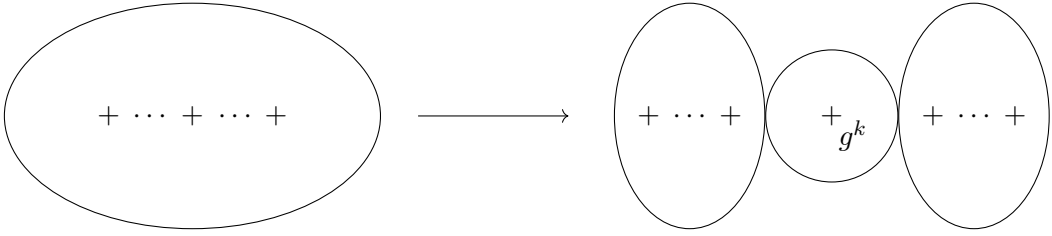
\begin{figure}[h!]
  \begin{tikzpicture}\centering

\begin{scope}
\node[draw, ellipse, minimum height = 3cm, minimum width = 5cm] (A) at (0,0) {$+\;\cdots \; + \; \cdots \; +$};

\draw[->] (3,0) -- (5,0);
\end{scope}

\begin{scope}[xshift = 8.5cm]

\node[draw, ellipse, text=white, minimum height = 3cm, minimum width = 2cm] (B) at (-1.89,0) {};
\node[draw,circle, minimum height= 50pt] (C) at (0,0) {} ;
\node[draw, ellipse, text=white, minimum height = 3cm, minimum width = 2cm] (D) at (1.89,0) {};

\node (glue) at (0,0) {$+$};
\node (gk) at (glue.south east) {$g^k$};

\node (leftdegeneration) at (B.center) {$+\;\cdots \; +$};
\node (rightdegeneration) at (D.center) {$+\; \cdots \; +$};

\end{scope}

\end{tikzpicture}
  \caption{Degeneration of the domain curve. The middle component is contracted to a point in $H_k^g$.}
  \label{fig:degeneration-domain-curve}
\end{figure}

We will need to keep track of the extension curve classes, so we introduce the following notations:
\begin{itemize}
  \item $\widehat{\delta}_V$ is the extension curve class corresponding to the extension from $V$\linebreak to $\widehat{V}$.
  \item Let $V_k$ be the infinitesimal cylinder of twig type $w_k$ of Construction \ref{constr:elementary-cylinders}, whose extension is $N_k$. We denote by $\widehat{\delta}_k$ the extension curve class.\\
  By construction, if we truncate the $1'$ and $2'$ legs of $M_k$ to finite legs whose image intersect at most one wall, then the associated extension curve class is precisely $\widehat{\delta}_k$.
  \item Consider the infinitesimal spine obtained by truncating every boundary leg of $L_{t+1}$ to a finite leg whose image intersects at most one wall. The corresponding extension curve class is $\widehat{\delta}_V + \sum_{s=1}^t {\delta}_s$, where ${\delta}_s$ is associated to the $t^s$-leg.
\end{itemize}

The next lemma computes the toric counts that appear in the inductive formula of Proposition \ref{prop:splitting-formula}.

\begin{lemma}\label{lemma:compute-M}
  For $1\leq k\leq t$ and $\beta\in \NE (Y)$, we have
  \[N(M_k , \beta ) = \left\lbrace\begin{matrix} 1 & \text{if} & \beta = \widehat{\delta}_k ,\\
  0 & \text{else} .&
  \end{matrix}\right. \]
\end{lemma}

\begin{proof}
  By construction of $W_k^M$, the moduli space ${M} (T_k^M ,\beta )$ is contained in the smooth locus $\sm{{M}} (U,\p_k^M ,\beta )$. Thus the evaluation maps $\ev_{w'}$ and $\ev_{g'}$ are étale \cite[Lemma 3.6]{keelFrobeniusStructureTheorem2019a}, and the invariant is the cardinality of the intersection of two fibers of these maps over arbitrary points.

  Let $F$ denote the fiber of $(\ev_{w'} , \ev_{g'})$ over the points $(x_{w'} ,x_{g^k} ) \in\base\subset U$ (recall that the essential skeleton is naturally included in $Y$). Any $f\in F$ is skeletal, meaning it has image contained in $\bbase$, by \cite[Theorem 8.18]{keelFrobeniusStructureTheorem2019a}.

  Let $\Gamma_k^M$ denote $\st (M_k $), we claim that $F \subset\st^{-1} (\Gamma_k^M)$. Indeed, for $f\in F$ the stablization $\st (f)$ is obtained by taking the convex hull of the marked points. Here, the image of $f$ is fixed, equal to the image of $M_k$, and the domain of $f$ does not have any nodes. This fixes the combinatorial type and the slope of $f$ on every edge of the domain of $f$. Then $\st (f)$ is completely determined by the choice of the length of the two finite edges. These lengths can be uniquely recovered from the image of the two interior marked points and the slopes of $f$.

  The previous observation implies that $F$ is contained in \linebreak ${(\st , \ev_{w'})^{-1} (\Gamma_k^M , x_{w'})\cap\Sp^{-1} (M_k )}$.
  Stable maps in this set do not meet the\linebreak exceptional locus $E$, so they correspond uniquely to stable map \linebreak in $\sm{{M}} (U_t , \p_k^M , \pi_{\ast}\beta )$.
  In the toric case, the map $(\st , \ev_{w'})$ is an open immersion with image ${M}_{0,5}\times U$ \cite[][Proposition 6.2]{keelFrobeniusStructureTheorem2019a}. Thus $F$ has cardinality at most $1$, and it is non empty if and only if $\beta=  \widehat{\delta}_k$, proving the lemma.
\end{proof}

\begin{prop}\label{prop:splitting-formula}
  Same notations as in Lemma \ref{lemma:compute-M}. For $\beta\in\NE (Y)$ and $1\leq k\leq t$, we have
  \[N(L_{k} , \beta - \widehat{\delta}_k )  = N(T_k^g ,\beta ) = \sum_{\beta = \beta_1+\beta_2} N(L_{k+1} ,\beta_1 )N(N_k , \beta_2) .\]
\end{prop}

\begin{figure}
  \begin{subfigure}{\textwidth}
  \begin{tikzpicture}
\tikzstyle{length} = [<->, dotted, thin]
\tikzstyle{leg} = []
\tikzstyle{twig} = [red]
\tikzstyle{interior-point} = [draw]

\begin{scope}[scale=0.6]
  \node[draw, circle] (1) at (-3,0) {} ;
  \node[draw,circle] (2) at (0,0) {} ;
  \node[draw,circle] (3) at (3,0) {} ;

  \draw (1.east) -- (2.west) ;
  \draw (2.east) -- node[above]{$e$} (3.west) ;

  \draw[] (2.south) -- (0,-2) node[below] {$g^k$};

  \draw[] (1.north west) -- (-4,1) ;
  \draw[] (1.south west) -- (-4,-1) ;
  \node[] (aux1) at (-4,0) {$\vdots$};

  \draw (3.north) -- (3,1.5) node[above]{$1'$} ;
  \draw (3.north east) -- (4,1) node[above right] {$2'$} ;
  \draw (3.south east) -- (4,-1) node[below right] {$w'$} ;
  \draw (3.south) -- (3 , -1.5) node[below] {$t'$};
.
  \node[] (A) at (0,-3.5) {$\tau$};

\end{scope}

\begin{scope}[scale=0.6, shift={(13,0)}]
\node[draw, circle] (1) at (-3,0) {} ;
\node[draw,circle] (2) at (0,0) {} ;
\node[draw,circle] (3) at (5,0) {} ;

\draw (1.east) -- (2.west) ;

\draw[] (2.south) -- (0,-2) node[below] {$g^k$};
\draw[] (2.east) -- (2,0) node[right] {$e_1$};

\draw[] (1.north west) -- (-4,1) ;
\draw[] (1.south west) -- (-4,-1) ;
\node[] (aux1) at (-4,0) {$\vdots$};

\draw (3.north) -- (5,1.5) node[above]{$1'$} ;
\draw (3.north east) -- (6,1) node[above right] {$2'$} ;
\draw (3.south east) -- (6,-1) node[below right] {$w'$} ;
\draw (3.south) -- (5 , -1.5) node[below] {$t'$};
\draw (3.west) -- (4,0) node[left] {$g'$};

\node[] (A) at (2,-3.5) {$\sigma_1\coprod \underline{\sigma_2}$};

\end{scope}

\end{tikzpicture}
  \subcaption{Cutting the edge $e$ of $\tau$ to obtain $\sigma$.}
  \end{subfigure}
  \vspace{1cm}

  \begin{subfigure}{\textwidth}
  \begin{tikzpicture}
\tikzstyle{length} = [<->, dotted, thin]
\tikzstyle{leg} = []
\tikzstyle{twig} = [red]
\tikzstyle{interior-point} = [draw]

\begin{scope}[scale=0.6, shift={(0,0)}]
\node[draw, circle] (1) at (-3,0) {} ;
\node[draw,circle] (2) at (0,0) {} ;

\draw (1.east) -- (2.west) ;

\draw[] (2.south) -- (0,-2) node[below] {$g^k$};
\draw[] (2.east) -- (2,0) node[right] {$e_1$};

\draw[] (1.north west) -- (-4,1) ;
\draw[] (1.south west) -- (-4,-1) ;
\node[] (aux1) at (-4,0) {$\vdots$};

\node[] (A) at (0,-3.5) {$\sigma_1$};

\end{scope}

\begin{scope}[scale=0.6, shift={(13,0)}]
\node[draw, circle] (1) at (-3,0) {} ;



\draw[] (1.north west) -- (-4,1) ;
\draw[] (1.south west) -- (-4,-1) ;
\node[] (aux1) at (-4,0) {$\vdots$};
\draw[] (1.east) --  (0,0) node[right]{$g^k$} ;

\node[] (A) at (0,-3.5) {$\underline{\sigma_1}$};

\end{scope}

\end{tikzpicture}
  \subcaption{Forgetting the tail $e_1$ to obtain $\sigma'$.}
  \end{subfigure}
  \caption{Degeneration of the combinatorial type.}
  \label{fig:combinatorics-degeneration}
\end{figure}
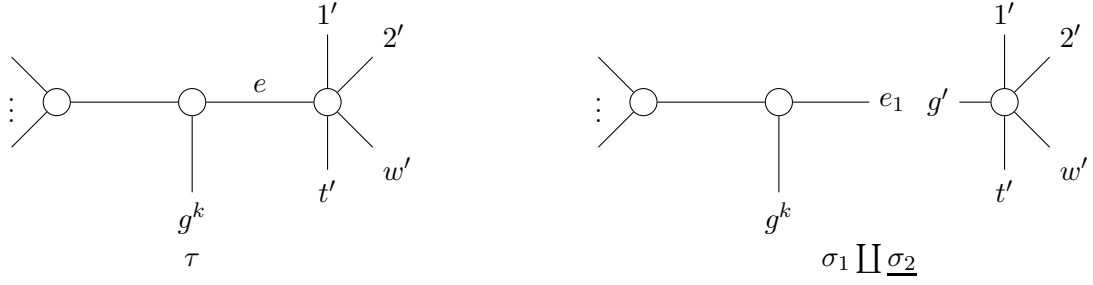
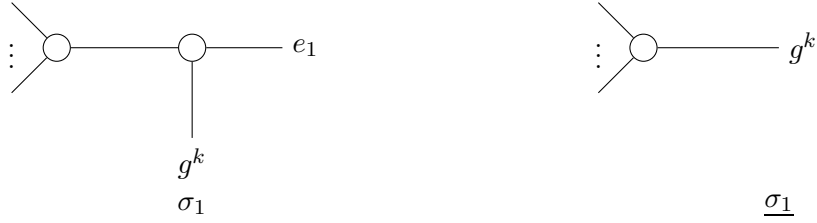

\begin{proof}
  Let us prove the first equality. We fix $\mu\in\overline{{M}}_{0,5}$ given by the partition $(w,t^k \vert g^k \vert w', t')$, and let ${M} (T_k^g,\beta )_{W,\mu}$ denote the substack of stable maps over $\mu$. For $i\in I_k^g$, we denote by $\gamma_i\in H^{\ast} (X, \Q (\ast ))$ a cycle represented by a smooth subvariety $Z_i$ of $Y$. We assume $Z_{g^k} = H_k^g$.

  We shall use $(\tau,\beta )$-marked stable maps to keep track of the combinatorics of the degeneration \cites{behrend_stacks_1995}{porta_non-archimedean_2022}{porta_non-archimedean_2022-1}.
  By construction, stable maps in ${M} (T_k^g , \beta )_{W,\mu}$ have a fixed graph type $\tau$.
  Fix a decomposition of $\beta$ into effective curve classes ${\beta = \beta_1 + \beta_2}$. Denote by $\underline{\tau} = (\tau , (\beta_1, 0,\beta_2))$ the associated $A$-graph.

  Consider the moduli space
  \[{M}^g (\underline{\tau} )\coloneqq {M} (T_k^g ,\beta )_{W,\mu}\bigcap_{i\in I_k^g} \ev_i^{-1} (Z_i) \cap\overline{{M}} (Y , \underline{\tau}), \]
  denote by $q\colon {M}^g (\underline{\tau})\rightarrow \pt$ its structure morphism.
  Let $\underline{\sigma} = \underline{\sigma_1}\coprod\underline{\sigma_2}$ denote the new $A$-graph obtained by cutting the edge connecting the middle vertex to the rightmost vertex and forgetting the new tail (Figure \ref{fig:combinatorics-degeneration}). Consider the moduli spaces
  \[{M}^{1} (\underline{\sigma_1} ) \coloneqq {M} (T_k^L , \beta_1)_W \bigcap_{i\in I_k^L} \ev_i^{-1} (Z_i) \cap \overline{{M}} (Y , \underline{\sigma_1}) ,\]
  \[{M}^2 (\underline{\sigma_2} ) \coloneqq {M} (T_k^M , \beta_2)_W \bigcap_{i\in I_k^M} \ev_i^{-1} (Z_i) \cap \overline{{M}} (Y , \underline{\sigma_2}) .\]
  Let ${M}^{\text{cut}} (\underline{\sigma})$ be the moduli space defined by the derived pullback diagram:
  \begin{cd}
    {M}^{\text{cut}} (\underline{\sigma}) \ar[rd,"q'",no head] \ar[r, "p_2"] \ar[d , "p_1"] & {M}^2 (\underline{\sigma_2}) \ar[rdd, bend left, "q_2"] \ar[d,"\ev_{g'}"] \\
    {M}^1 (\underline{\sigma_1}) \ar[r, "\ev_{g^k}"] \ar[rrd, bend right, "q_1"]& Z_{g^k} \ar[rd] \\
    & & \pt
  \end{cd}
  Let $c\colon{M}^g (\underline{\tau})\rightarrow {M}^{1} (\underline{\sigma_1})\times {M}^2 (\underline{\sigma_2} ) $ denote the composition of the morphism cutting the edge $e$ and forgetting the new tail $e_1$. Note that in this case, at the level of domain curves the map forgetting the tail $e_1$ is an isomorphism.
  From Proposition \ref{prop:properness-restriction}.(4) we deduce that the map $c$ induces an isomorphism ${{M}^g (\underline{\tau}) \overset{\sim}{\rightarrow}{M}^{\text{cut}} (\underline{\sigma})}$.

  The map $c$ is the composition of a cutting-an-edge morphism and a forgetting-a-tail morphism, so it is compatible with virtual fundamental classes by \cite[][Theorem 1.1]{porta_non-archimedean_2022-1}. Together with Lemma \ref{lemma:splitting-fundamental-class} this gives
  \[q_{\ast} [{M}^g (\underline{\tau})] = q_{\ast}' [{M}^{\text{cut}} (\underline{\sigma})] = q_{1\ast} [{M}^1 (\underline{\sigma_1})] \cdot q_{2\ast} [{M}^2(\underline{\sigma_2}) / Z_{g^k} ] .\]
  Now we specialize to $Z_i$ a smooth point for $i\in \lrbrace{w,w',g^1,\cdots ,g^{k-1}}$, $Z_{g^{\ell}} = H_{\ell}^g$ and $Z_{t^{\ell}} = H_{\ell}^t$ for $k\leq \ell\leq t$.

  By Lemma \ref{lemma:relative-fundamental-class} and Lemma \ref{lemma:restriction-fundamental-class}
  \begin{align*}
  q_{2\ast} [{M}^2 (\underline{\sigma_2}) / H_k^g] &= q_{\ast}^M \lr{(\ev_{w'}^{\ast} [\pt ]\cup \ev_{g'}^{\ast} [\pt ]\cup \ev_{t'}^{\ast} [H_k^t] )\cap [M (T_k^M , \beta_2 )_W] } \\
  &= N(M_k,\beta_2) .
  \end{align*}
  Similarly, Lemma \ref{lemma:restriction-fundamental-class} gives
  \[q_{1\ast} [{M}^1 (\underline{\sigma_1})] = N(L_k , \beta_1 ) .\]
  Given that the union over all $A$-graph $\underline{\tau}$ associated to a splitting $\beta = \beta_1+0+\beta_2$ into effective classes equals the moduli space responsible for the count $N(T_k^g , \beta )$, we deduce the first eqality
  \[N(T_k^g ,\beta )  = \sum_{\beta_1+\beta_2 = \beta} N(L_k ,\beta_1 ) N(M_k ,\beta_2 ) = N(L_k ,\beta - \widehat{\delta}_k)  .\]
  The last equality being obtained by Lemma \ref{lemma:compute-M}. A similar reasoning based on the choice of domain $\mu '\in\overline{{M}}_{0,5}$ given by the partition $(w,t'\vert g^k \vert w',t^k)$ proves the second equality.
\end{proof}

The previous proof relies on the following lemmas, which are direct computations making use of the very good functoriality properties of the analytic Borel-Moore motivic cohomology of derived analytic spaces. Lemma \ref{lemma:restriction-fundamental-class} expresses the\linebreak compatibility of the restriction of the virtual fundamental class to a derived lci subspace.
Lemma \ref{lemma:relative-fundamental-class} relate the relative virtual fundamental class to an absolute virtual fundamental class, and Lemma \ref{lemma:splitting-fundamental-class} is a splitting formula of the virtual fundamental class of a derived fiber product.

\begin{lemma} \label{lemma:restriction-fundamental-class}
  Consider a pullback square of derived $k$-analytic spaces over a point
  \begin{cd}
    W \ar[r,"i"] \ar[d , "g"] & X \ar[d,"f"] \\
    Y \ar[r, "j"] & Z
  \end{cd}
  Assume $j$ is derived lci and $Z$ is smooth. Then $\lr{f^{\ast} \PD^{-1} j_{\ast} [Y] } \cap [X] = i_{\ast} [W]$.
\end{lemma}

\begin{proof}
  First, we note that since $Z$ is smooth, we have $\gamma \cap [Z] = \PD (\gamma )$ for all cycle $\gamma$ \cite[][Theorem 2.26]{khan_virtual_nodate}. In particular, we get $\PD^{-1} j_{\ast} [Y]\cap [Z] = j_{\ast} [Y]$. Then
  \begin{align*}
    (f^{\ast} \PD^{-1} j_{\ast} [Y]) \cap [X] &= (f^{\ast} \PD^{-1} j_{\ast} [Y]) \cap f^![Z] \\
    &= f^! (\PD^{-1} j_{\ast} [Y]  \cap [Z] ) && \text{by \cite[][Proposition 4.10.(1)]{porta_non-archimedean_2022-1}} \\
    &= f^! j_{\ast} [Y] \\
    &= i_{\ast} g^! [Y] && \text{by \cite[][Proposition 4.11]{porta_non-archimedean_2022-1}}\\
    &= i_{\ast} [W] .
  \end{align*}
\end{proof}

\begin{lemma} \label{lemma:relative-fundamental-class}
  Consider a pullback square of derived $k$-analytic spaces over a point
  \begin{cd}
    W \ar[r,"i"] \ar[d,"g"]& X \ar[d,"f"] \ar[rdd, "q", bend left]\\
    \pt \ar[rrd, equal, bend right]\ar[r,"j"] & Z \ar[rd,"p"] \\
    & & \pt
  \end{cd}
  Assume $j$ is a smooth point of $Z$, and $f$ and $p$ are proper. Let $\gamma\in H^{\ast} (X , \ca{F})$, then
  \[g_{\ast} (i^{\ast }\gamma \cap [W]) = q_{\ast} (\gamma \cap [X/Z]) .\]
\end{lemma}

\begin{proof}
This is a consequence of the projection formula
\begin{align*}
  g_{\ast} (i^{\ast} \gamma \cap [W]) &= q_{\ast} i_{\ast} (i^{\ast}\gamma \cap [W]) \\
  &=q_{\ast} (\gamma\cap i_{\ast} [W]) && \text{by the projection formula}\\
  &= q_{\ast} (\gamma\cap i_{\ast} i^{\ast} [X/Z]) &&\text{by base change \cite[][Proposition 4.7]{porta_non-archimedean_2022-1}} \\
  &= q_{\ast} (\gamma \cap [X/Z]) && \text{because } i \text{ is an immersion.}
\end{align*}
\end{proof}

\begin{lemma} \label{lemma:splitting-fundamental-class}
Consider a pullback square of derived $k$-analytic spaces over a point
\begin{cd}
  W \ar[r,"k"]\ar[rd,"e"] \ar[d,"g"] & X \ar[d,"f"] \ar[rdd,"b", bend left] \\
  Y \ar[r,"h"] \ar[rrd,"c", bend right = 20] & Z \ar[rd,"d"] \\
  & & \pt
  \end{cd}
  Let $a=d\circ e$. Assume $d,f$ and $h$ are proper and derived lci. Then in $\BMho[0] (\pt, \ca{F})$
  \[a_{\ast} [W] = b_{\ast} [X] c_{\ast} [Y/Z] .\]
\end{lemma}

\begin{proof}
Recall the compatibility between pushforward and composition product \cite[][\S2.3.4]{khan_virtual_nodate} for a proper map $f\colon X\rightarrow Y$ of derived $k$-analytic spaces over $S$, for all ${\alpha\in \BMho[s] (X/Y ,\ca{F} (r))}$ and $\beta\in\BMho [s'] (Y/S  , \ca{F} (r') )$ we have $f_{\ast} (\alpha \circ\beta ) = f_{\ast} (\alpha ) \circ \beta$.

We also recall that for a derived $k$-analytic space $X$, in $\BMho (X/X , \ca{F}(\ast ))$ the right composition product, the left composition product and the external product $\boxtimes$ coincide. Furthermore, under the identification with motivic cohomology\linebreak ${\BMho[s] (X/X , \ca{F} (r)) = \H^{-s} (X,\ca{F} (-r))}$ these coincide with the cup-product of motivic cohomology on Borel-Moore homology, and with the cap-product of\linebreak motivic cohomology. In particular, the composition product becomes commutative in this case.

We now compute
\begin{align*}
  a_{\ast} [W] &= a_{\ast} \lr{[W/Z]\circ [Z]} && \text{by \cite[][Proposition 4.6]{porta_non-archimedean_2022-1}} \\
    &= d_{\ast} e_{\ast} \lr{ \lr{[X/Z]\extprod [Y/Z]}\circ [Z]} && \text{by \cite[][Propositions 4.6, 4.7 and 4.10.(4)]{porta_non-archimedean_2022-1}} \\
  &=  d_{\ast} \lr{e_{\ast} \lr{[X/Z]\extprod [Y/Z]}\circ [Z] } && \text{since } e_{\ast} (\alpha \circ\beta )= e_{\ast} (\alpha) \circ\beta \\
   &= d_{\ast} \lr{\lr{f_{\ast} [X/Z] \extprod h_{\ast} [Y/Z]}\circ [Z] } && \text{by \cite[][Proposition 4.10.(6)]{porta_non-archimedean_2022-1}} \\
  &= d_{\ast} \lr{f_{\ast} [X/Z] \circ h_{\ast} [Y/Z]\circ [Z]}  && \text{since } \extprod = \circ \\
  &= d_{\ast} \lr{f_{\ast} [X/Z] \circ [Z] \circ h_{\ast} [Y/Z] } && \text{by commutativity} \\
  &= d_{\ast} \lr{f_{\ast}\lr{ [X/Z] \circ [Z] } \circ h_{\ast} [Y/Z] } && \text{since } f_{\ast} (\alpha \circ\beta )= f_{\ast} (\alpha) \circ\beta \\
  &= d_{\ast} \lr{f_{\ast} [X] \circ h_{\ast} [Y/Z] } && \text{by \cite[][Proposition 4.6]{porta_non-archimedean_2022-1}}\\
  &= b_{\ast} [X] c_{\ast} [Y/Z] && \text{by \cite[][Proposition 4.10.(6)]{porta_non-archimedean_2022-1}}.
\end{align*}

\end{proof}

The next proposition expresses the counts that appear in the initial step and the final steps of the inductive twig-removal procedure.

\begin{prop}\label{prop:compute-L}
  Given $\beta\in\NE (Y)$, we have:
  \begin{enumerate}
    \item $N (L_1 ,\beta ) = N(\widehat{V},\beta )$, the count of the extended initial cylinder.
    \item $N(L_{t+1} ,\beta ) =0$ for $\beta\neq \widehat{\delta}_V+\sum_{s=1}^t \delta_s$, and $N(L_{t+1} ,\widehat{\delta}_V+\sum_{s=1}^t \delta_s )=1$.
  \end{enumerate}
\end{prop}

\begin{proof}
  For (1), we note that the count $N(L_1,\beta )$ is given by imposing divisorial conditions at the marked points corresponding to indices in $I_1^L\setminus \lrbrace{w}$. \linebreak By construction, the divisors $H_k^g$ and $H_k^t$ have intersection number $1$ with $\beta$. Thus, applying repeatedly the divisor axiom we see that $N(L_1 ,\beta )$ is given by a count of $3$-pointed curves with two boundary marked points and one interior marked point lying above the initial tropical cylinder. As we evaluate at the interior marked point, we get the count of the initial cylinder by definition.

  For (2), note that $N(L_{t+1}, \beta )$ counts curves without any twigs -- \emph{i.e.} curves that do not meet the exceptional divisor $E$. We can argue as in the proof of Lemma \ref{lemma:compute-M}: evaluate $N(L_{t+1} , \beta )$ as the cardinality of $F = \ev_{I_{t+1}^g}^{-1} ((x_i)_{i\in I_{t+1}^g} )$ which consists of skeletal curves, prove that the domain of an element in $F$ is completely determined by the image of the interior points, and use the result on toric spine counts.
\end{proof}

We can now prove our main result.

\begin{proof}[Proof of Theorem \ref{theo:reduction-primitive-twig}]
Applying Proposition \ref{prop:splitting-formula} inductively, and using Proposition \ref{prop:compute-L}, for $\beta\in\NE (Y)$ we compute
\begin{align*}
  N(V,\beta ) &= N(\widehat{V} , \beta + \widehat{\delta}_V) \\
  &= N(L_1 ,\beta + \widehat{\delta}_V) \\
  &= \sum_{\gamma_1+\cdots +\gamma_t +\gamma_{t+1} = \beta +\widehat{\delta}_V + \sum_{s=1}^t {\delta}_s} N\lr{L_{t+1} , \gamma_{t+1} - \sum_{s=1}^t \widehat{\delta}_s} \prod_{s=1}^t N(N_s, \gamma_s) \\
  &= \sum_{\gamma_1+\cdots +\gamma_t = \beta +\sum_{s=1}^t \widehat{\delta}_s } \prod_{s=1}^t N(N_s, \gamma_s) \\
  &= \sum_{\gamma_1+\cdots +\gamma_t = \beta +\sum_{s=1}^t \widehat{\delta}_s} \prod_{s=1}^t N(V_s, \gamma_s - \widehat{\delta}_s) \\
  &= \sum_{\beta_1+\cdots +\beta_t = \beta}\prod_{s=1}^t N(V_s, \beta_s).
\end{align*}
\end{proof}

\printbibliography[env=bibliography]
\end{document}